\newif\ifArxiv

\Arxivtrue

\ifArxiv
\documentclass[journal,twoside,web]{ieeecolor12_arxiv}              
\else
\documentclass[journal,twoside,web]{ieeecolor12}              
\fi

\ifArxiv

\usepackage[T1]{fontenc}
\usepackage{amsmath,amssymb,amsfonts,mathtools}
\usepackage{mathrsfs}
\usepackage{dsfont}
\usepackage{bbm}
\usepackage{bbold}
\usepackage{textcomp}
\usepackage{siunitx}
\usepackage{graphicx}
\usepackage[ruled]{algorithm2e}

\let\labelindent\relax
\usepackage{enumitem}

\usepackage{caption}
\usepackage{subcaption}
\usepackage[dvipsnames]{xcolor} 
\definecolor{thmBgColor}{HTML}{E8F0F8}
\definecolor{assumBgColor}{HTML}{EAF2E8}
\definecolor{highlightBgColor}{rgb}{0.98, 0.94, 0.88}

\definecolor{niceGreen}{rgb}{0.1, 0.625, 0.1}
\definecolor{cn}{RGB}{93,147,191}      
\definecolor{cz}{RGB}{233,  72, 73}      
\definecolor{cp}{RGB}{113, 191, 110}

\DeclarePairedDelimiterX{\inp}[2]{\langle}{\rangle}{#1, #2}

\usepackage{tikz}
\usetikzlibrary{arrows, arrows.meta, graphs}
\usepackage{pgfplots}
\pgfplotsset{compat=1.18}
\usepackage{pgf-pie}
\usepackage{etoolbox}

\newtoggle{showpct}
\makeatletter
\patchcmd{\pgfpie@slice}%
{\scalefont{#3}\beforenumber#3\afternumber}%
{\iftoggle{showpct}{\scalefont{#3}\beforenumber#3\afternumber}{}}%
{}{}
\makeatother

\usepackage{amsthm}
\renewcommand{\qedsymbol}{\hfill $\blacksquare$} 
\theoremstyle{definition}

\newtheorem{theorem}{Theorem}[section]

\newtheorem{lemma}[theorem]{Lemma}
\newtheorem{proposition}[theorem]{Proposition}
\newtheorem{remark}[theorem]{Remark} 
\newtheorem{corollary}[theorem]{Corollary}
\newtheorem{definition}[theorem]{Definition}

\newtheorem{example}[theorem]{Example}
\newtheorem{assumption}{Assumption}
\newtheorem{fact}[theorem]{Fact}

\newtheorem{thm}[theorem]{Theorem}
\newtheorem{lem}[theorem]{Lemma}
\newtheorem{prop}[theorem]{Proposition}
\newtheorem{cor}[theorem]{Corollary}
\newtheorem{defn}[theorem]{Definition}
\newtheorem{conj}[theorem]{Conjecture}
\newtheorem{exmp}[theorem]{Example}
\newtheorem{rmk}[theorem]{Remark}
\newtheorem{ass}[assumption]{Assumption}

\usepackage[breakable]{tcolorbox} 
\tcbset{
  mathbox/.style={
    breakable, colback=thmBgColor, colframe=thmBgColor,
    boxrule=0pt, arc=4pt, top=1ex, bottom=1ex, left=1ex, right=1ex,
    before skip=1em, after skip=1em
  },
  assumbox/.style={
    breakable, colback=assumBgColor, colframe=assumBgColor,
    boxrule=0pt, arc=4pt, top=1ex, bottom=1ex, left=1ex, right=1ex,
    before skip=1em, after skip=1em
  }
}

\tcolorboxenvironment{theorem}{mathbox} \tcolorboxenvironment{thm}{mathbox}
\tcolorboxenvironment{lemma}{mathbox}   \tcolorboxenvironment{lem}{mathbox}
\tcolorboxenvironment{proposition}{mathbox} \tcolorboxenvironment{prop}{mathbox}
\tcolorboxenvironment{corollary}{mathbox} \tcolorboxenvironment{cor}{mathbox}

\tcolorboxenvironment{definition}{assumbox} \tcolorboxenvironment{defn}{assumbox}
\tcolorboxenvironment{assumption}{assumbox} \tcolorboxenvironment{ass}{assumbox}

\newtcolorbox{texthighlight}{
  breakable, colback=highlightBgColor, colframe=highlightBgColor,
  arc=4pt, boxrule=0pt, top=1ex, bottom=1ex, left=1ex, right=1ex,
  before skip=1ex, after skip=1ex, after={\par\noindent\ignorespaces}
}

\makeatletter
\@ifpackageloaded{etoolbox}{
  \AfterEndEnvironment{texthighlight}{\noindent\ignorespaces}
}{}
\makeatother
\else

\usepackage{cite}
\usepackage{amsmath,amssymb,amsfonts}
\newcommand{\qedsymbol}{\hfill $\blacksquare$}
\usepackage{graphicx}
\usepackage{textcomp}
\usepackage{mathtools}
\usepackage{bbm}
\usepackage{empheq}
\usepackage[T1]{fontenc}
\usepackage{siunitx}
\usepackage{dsfont}
\usepackage[ruled]{algorithm2e}
\usepackage{mathrsfs}
\usepackage{tikz}
\usetikzlibrary{arrows}
\usetikzlibrary{arrows.meta,graphs}

\usepackage{bbold}
\usepackage{fancyhdr}
\usepackage{pgfplots}
\let\labelindent\relax
\usepackage{enumitem}   

\definecolor{niceGreen}{rgb}{0.1, 0.625, 0.1}

\DeclarePairedDelimiterX{\inp}[2]{\langle}{\rangle}{#1, #2}

\definecolor{cn}{RGB}{93,147,191}           
\definecolor{cz}{RGB}{233,  72, 73}      
\definecolor{cp}{RGB}{113, 191, 110}    

\newtheorem{thm}{Theorem}[section]
\newtheorem{lem}[thm]{Lemma}
\newtheorem{prop}[thm]{Proposition}
\newtheorem{cor}{Corollary}
\newtheorem{defn}{Definition}[section]
\newtheorem{conj}{Conjecture}[section]
\newtheorem{exmp}{Example}[section]
\newtheorem{rmk}{Remark} 
\newtheorem{ass}{Assumption}

\usepackage{pgf-pie}
\usepackage{etoolbox}
\newtoggle{showpct}
\makeatletter
\patchcmd{\pgfpie@slice}%
{\scalefont{#3}\beforenumber#3\afternumber}%
{\iftoggle{showpct}{\scalefont{#3}\beforenumber#3\afternumber}{}}%
{}{}
\makeatother

\usepackage{subcaption}

\fi
\usepackage{generic}                             

\usetikzlibrary{arrows.meta,bending,positioning}
\usetikzlibrary{calc}
\usetikzlibrary{patterns}
\newcommand\mydots{\hbox to 1em{.\hss.\hss.}}
\usetikzlibrary{positioning,shapes.geometric,fit}

\newcommand{\densedots}{.\kern0.05em.\kern0.05em.}

\usepackage{bm}
\usepackage{booktabs}
\usepackage{multirow}

\makeatletter
\let\NAT@parse\undefined 
\expandafter\let\csname ver@cite.sty\endcsname\relax
\makeatother

\ifArxiv
\usepackage{cite}  
\fi

\definecolor{safegreen}{HTML}{009E73}
\definecolor{unsafeorange}{HTML}{D55E00}
\definecolor{safeblue}{HTML}{0072B2}

\pgfplotsset{compat=1.17}

\ifArxiv
\makeatletter

\let\oldtitle\title
\renewcommand{\title}[1]{\oldtitle{\color{black}#1}}

\def\ps@headings{%
  \def\@oddhead{}%
  \def\@evenhead{}%
  \def\@oddfoot{\small\rmfamily\hfil\thepage\hfil}%
  \def\@evenfoot{\small\rmfamily\hfil\thepage\hfil}%
}

\def\@IEEEheaderstyle{}

\makeatother

\AtBeginDocument{
  \colorlet{spotcolor}{black}
  \colorlet{journalcolor}{black}
  \colorlet{secnumcolor}{black}
  \colorlet{jnlcolor}{black}
  \colorlet{color1}{black}
  \colorlet{color2}{black}
  \colorlet{accessblue}{black}
  
  \colorlet{subsectioncolor}{black}
  \colorlet{nblue}{black}
  \colorlet{mblue}{black}
}
\fi

\usepackage{hyperref}
\hypersetup{
    hidelinks
}

\begin{document}
\title{Joint Chance Constrained Safe-Optimal Control}
\author{Niklas Schmid, Jared Miller, Tristan Zeller, Marta Fochesato, Tobias Sutter, John Lygeros
\thanks{N.~Schmid, M.~Fochesato, and J.~Lygeros are with the Automatic Control Laboratory (IfA), ETH Z\"urich, 8092 Z\"urich, Switzerland, 
        {\tt\small \{nikschmid, jlygeros\}@ethz.ch}, {\tt\small \ marta.fochesato@gmail.com}}%
\thanks{J.~Miller is with the Institute of Mathematical Methods in Engineering, Numerical Analysis and Geometric Modeling, University of Stuttgart, 70569 Stuttgart, Germany,
        {\tt\small \ jared.miller@imng.uni-stuttgart.de}}%
\thanks{T.~Zeller is with the Gymnasium Burgdorf, 3400 Burgdorf,
        {\tt\small \ trzeller@student.ethz.ch}}%
\thanks{T.~Sutter is with the Swiss Institute for Empirical Economic Research (SEW), University of St.Gallen, 9000 St.~Gallen, Switzerland
        {\tt\small \ tobias.sutter@unisg.ch}}%
\thanks{Work supported by the Swiss National Science Foundation under NCCR Automation under grant 51NF40\_225155.}%
}

\maketitle
\thispagestyle{headings}


\begin{abstract}
We consider the finite-time optimal control of stochastic systems subject to a probabilistic constraint on the trajectories' safety. Such formulations are known as joint chance constrained optimal control problems. The common practice is to jointly minimise the expected cost of all trajectories, safe and unsafe. This leads to policies which invite constraint violations to exploit low-cost unsafe trajectories. When constraints represent states of critical failure, such behaviour is undesirable. We demonstrate that this behaviour can be overcome by only minimizing the expected cost of safe trajectories. The underlying rationale follows a practical intuition: In many applications, the cost incurred by unsafe trajectories is irrelevant (e.g., the battery usage of a crashed quadcopter), and one is usually interested in minimizing the cost of trajectories that are safe. We show that this problem can be cast as a constrained Markov Decision Process over an augmented state space. This allows solving it via dynamic programming. We derive bounds on the policies' safety under errors resulting from gridding approximations when the system's state space is continuous. Finally, we empirically compare dynamic programming as well as reinforcement learning solutions on a simulated 2D unicycle system in cluttered reach-avoid environments.
\end{abstract}

\section{Introduction}
\label{sec_intr}
\begin{figure}[!t]
    \centering
    \begin{subfigure}{0.49\linewidth}
        \centering
        {\small
    \begin{align*}
        &\text{JCC Optimal Control:} \\
        &
        \begin{aligned}
        &\ \min &&\text{Cost}(\textcolor{safegreen}{\blacksquare}+\textcolor{unsafeorange}{\blacksquare}) \\
        & \ \ \text{s.t.}  &&\mathbb{P}(\textcolor{safegreen}{\blacksquare}) \geq \alpha
    \end{aligned}
    \end{align*}}\includegraphics[width=\linewidth]{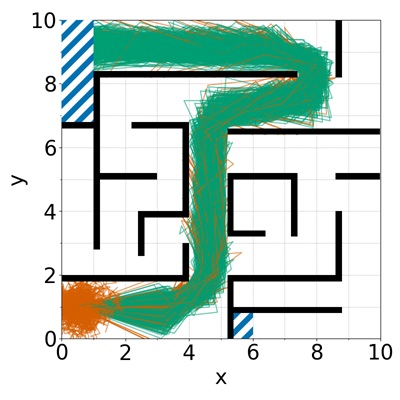}
    \end{subfigure}
    \begin{subfigure}{0.49\linewidth}
        \centering
        {\small
            \begin{align*}
            &\quad \ \ \ \text{JCC \emph{Safe}-Optimal Control:} \\
            &
            \begin{aligned}
                &\qquad \  \min &&\text{Cost}(\textcolor{safegreen}{\blacksquare}) \\
                & \qquad \ \ \text{s.t.}  &&\mathbb{P}(\textcolor{safegreen}{\blacksquare}) \geq \alpha
            \end{aligned}
        \end{align*}}
        \includegraphics[width=\linewidth]{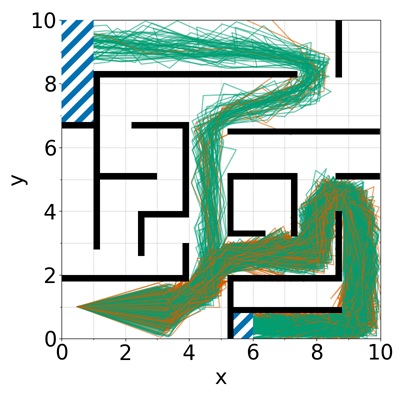}
    \end{subfigure}
    \caption{Trajectories of a perturbed delivery robot. The aim is to minimise actuation costs while safely reaching the target with probability $\alpha= 0.6$. The standard joint chance constrained control policy from the literature (left) achieves a low cost by occasionally "giving up" on the delivery objective, and following a long but safe route otherwise. In contrast, we exclusively optimize the cost of trajectories that safely reach the target (right), and disregard the cost of all others, leading to a lower actuation cost for those successful trajectories ($21.8$ left vs. $19.7$ right over $1000$ Monte-Carlo simulations), see Section~\ref{sec_num} for details. Black: Unsafe Set; Blue striped: Target Set; Trajectories are green if they reach the target, red otherwise; Initial state: $(0.5,1)$.}
    \label{fig_maze_trajectories}
\end{figure}

Safe autonomy under stochastic uncertainty is a fundamental challenge in control, robotics and machine learning: Autonomous driving \cite{yurtsever2020survey, rahman2025trajectory} promises fast but safe travel in hard-to-predict traffic environments, precision agriculture \cite{shang2019robust} aims for resource efficient irrigation with low risks in crop losses under uncertain weather predictions, and medical devices \cite{hewing2021volume} must sustain life-critical performance under tight safety constraints despite biological variability. Accordingly, stochastic control problems with joint chance constraints (JCC) are common in the control literature. Typical formulations comprise two components, a cost assigned to each trajectory (which, for the sake of the argument, we consider to be financial), and a constraint encoded by a set of safe trajectories (joint chance constraint). The aim of the control policy is to minimize the financial cost while ensuring that the probability that the closed loop trajectories remain safe is above a specified value. Solution approaches to JCC problems range from stochastic model predictive control \cite{farina_1, wang_2, bavdekar, raghuraman, paulson, ono_5}, constrained Markov Decision Processes (MDPs) solved through dynamic programming (DP) \cite{schmid2023computing, Ono_2, Wang, hahn2019interval}, linear programming \cite{schmid2024joint, haesaert2021formal, etessami2008multi,grontas2024operator}, and reinforcement learning (RL) \cite{chen_1, xu_1, zhang_1, ding_1, tessler_1, ni2025learning}, as well as Lyapunov-based methods \cite{mestres2025probabilistic}. 

In many applications, the violation of constraints is reversible. For instance, comfort constraint violations in buildings (\SI{26}{\celsius} rather than 
the nominal \SI{23}{\celsius} in the summer) are not a source of major concern and the constraints can be reentered after 
violation. If unsafe trajectories have a lower cost, this leads to a trade-off between constraint violation and cost: The resulting controller will ensure safety with the specified probability but take advantage of financial savings, potentially by violating constraints.

In other applications, however, constraint violation is catastrophic and irreversible. For example, consider an airline optimizing flight paths for minimum fuel while ensuring a sufficient level of safety. In such cases, trading-off safety for cost is meaningless; the fuel cost of trajectories that violate the constraint by crashing the airplane is irrelevant and should not be included in the operating profit of an airline. The fact that JCC control policies systematically exploit low-cost unsafe trajectories has been reported in the literature \cite{schmid2023computing,schmid2024joint,ni2025learning,grontas2024operator} and makes their applicability to systems with irreversible constraints unclear. 

Motivated by these reports, we study the problem of only optimizing safe trajectories in the financial cost and ignoring the cost of unsafe trajectories. Fig.~\ref{fig_maze_trajectories} displays how this modification affects the behaviour of JCC control policies on a delivery robot that minimizes battery usage. The standard JCC control policy occasionally "gives up" on the delivery-constraint to achieve a low-cost trajectory. In contrast, our approach achieves lower costs on safe trajectories at the risk of high-cost unsafe trajectories. We summarize the contributions of this paper:    
\begin{itemize}
    \item We introduce a novel control problem where costs of safe trajectories are minimized subject to a JCC.
    \item We prove that the problem can be reformulated as a constrained MDP over an augmented state space, whose optimal policies coincide with those of the original problem.
    \item We derive bounds on the policies' safety under errors from gridding based approximations.
    \item We empirically validate our DP method in numerical experiments and compare with a state-of-the-art alternative based on RL.
\end{itemize}

\textbf{Outline.}
We introduce necessary preliminaries and our problem formulation in Section \ref{sec_prelim}, cast it as a constrained MDP and derive DP approximation error bounds in Section \ref{sec_methods}, and numerically evaluate our formulation in Section \ref{sec_num}. We conclude with a summary in Section \ref{sec_concl}.

\textbf{Notation.} 
We denote by $[N]$ the set $\{0,1,\mydots,N\},N\in\mathbb{N}$. The indicator function of a set $\mathcal{A}$ yields $\mathbb{1}_{\mathcal{A}}(x)=1$ if $x\in {\mathcal{A}}$ and $\mathbb{1}_{\mathcal{A}}(x)=0$ otherwise. The set-wise difference of $\mathcal{A}$ and $\mathcal{X}$ is $\mathcal{A}\setminus \mathcal{X} = \{x\in \mathcal{A}: x\notin  \mathcal{X}\}$. For a topological space $S$, we denote by $C_b(S)$ the space of continuous bounded real-valued functions on $S$ and by $\mathcal{B}(S)$ the Borel $\sigma$-algebra of $S$. Further, $\boldsymbol{M}_+$ and $\mathcal{P}$ are the sets of finite non-negative measures and probability measures on $(S,\mathcal{B}(S))$. The Dirac-delta distribution centered at $s\in S$ is $\delta_{s}(\cdot)$. 
For a function $f$ with domain $S$ we denote the $\mathcal{L}_{\infty}$-norm as $\|f\|_{\infty} = \sup_{s \in S} |f(s)|$. 

\section{Preliminaries and Problem Formulation}
\label{sec_prelim}
We introduce three optimal control problems. Equation \eqref{eq_jcc_in_words} defines a joint-chance constrained MDP, which we generalize in equation \eqref{eq_main_problem_formulation_using_expectation} to optimize over safe trajectories only. In Section \ref{sec_methods_cmpd} we describe how to rewrite \eqref{eq_main_problem_formulation_using_expectation} as a constrained MDP of form \eqref{eq_standard_constrained_mdp}, which is solvable through DP. 

\textbf{Markov Decision Process.} We refer to \cite{hernandez2012discrete} for a detailed treatment of MDPs. An MDP over a finite time horizon $N$ is a tuple $\mathcal{M}=(\mathcal{X}, \mathcal{U}, \mathcal{Q})$, where the state space $\mathcal{X}$ and the input space $\mathcal{U}$ are Borel subsets of complete separable metric spaces equipped with the $\sigma$-algebras $\mathcal{B}(\mathcal{X})$ and $\mathcal{B}(\mathcal{U})$, respectively. Given a state $x_k\!\in\!\mathcal{X}$ and an input $u_k\!\in\!\mathcal{U}$, the stochastic kernel $\mathcal{Q}\!:\!\mathcal{B}(\mathcal{X})\times\mathcal{X}\times\mathcal{U}\!\rightarrow\![0,1]$ describes the stochastic state evolution, leading to $x_{k+1}\sim \mathcal{Q}(\cdot|x_k,u_k)$.

For $k \in[N-1]$ we define the Borel space of histories up to time $k$ recursively as $\mathcal{H}_k = \mathcal{X}\times\mathcal{U} \times \mathcal{H}_{k-1}$, with $\mathcal{H}_0 = \mathcal{X}$; a generic element $h_k \in {\mathcal{H}_k}$ is of the form $h_k = (x_k, u_{k-1}, \mydots, x_1, u_0, x_0)$ \cite{hernandez2012discrete}.

A policy is a sequence $\pi=(\pi_0,\mydots,\pi_{N-1})$ of Borel-measurable stochastic kernels $\pi_k$, that, given $h_k$, assign a probability measure $\pi_k(\cdot|h_k)$ on the set $\mathcal{B}(\mathcal{U})$. We denote the set of policies by $\Pi$. A policy is called Markov if for all $k\in[N]$ $\pi_k(B|h_k)$ only depends on the last state, $x_k$, in the history for all $B\in\mathcal{B}(\mathcal{U})$; by abuse of notation we write $\pi_k(B|h_k)=\pi_k(B|x_k)$ in this case. Given a policy $\pi\in\Pi$ and an initial condition $x_0\in\mathcal{X}$, a unique probability measure $\mathbb{P}_{x_0}^\pi$ of histories is defined over $\mathcal{B}(\mathcal{H}_N)$ \cite{hernandez2012discrete}. We denote the expectation over $\mathbb{P}_{x_0}^\pi$ by $\mathbb{E}_{x_0}^\pi[\cdot]$.

\textbf{Joint Chance Constrained Optimal Control.} 
We consider the problem of finding policies for MDPs which minimize costs described by measurable functions $\ell_k\!:\!\mathcal{X}\!\times\mathcal{U}\!\rightarrow\![0,\infty), k\in[N-1]$ and $\ell_N\!:\!\mathcal{X}\!\rightarrow\![0,\infty)$ which we refer to as stage and terminal cost, respectively. Further, we require that the state remains within a safe set of states $\mathcal{X}^s\in\mathcal{B}(\mathcal{X})$. We denote the set of safe histories by $\mathcal{H}_N^s=\{h_N\in\mathcal{H}_N: x_0,\dots,x_N\in\mathcal{X}^s\}$, and write that a policy has a safety of $\alpha$ if 
\begin{align*}
    \mathbb{P}_{x_0}^{\pi}\left(x_0,\mydots,x_N\in\mathcal{X}^s\right)=\mathbb{E}_{x_0}^\pi\left[\prod_{k=0}^N\mathbb{1}_{\mathcal{X}^s}(x_k)\right]\geq \alpha.
\end{align*} 
A JCC optimal control problem minimizes costs subject to a high enough safety \cite{schmid2023computing}
\begin{subequations}
    \label{eq_jcc_in_words}
    \begin{align}
        &\inf_{\pi\in\Pi} &&\mathbb{E}_{x_0}^{\pi} \left[\ell_N(x_N)+\sum_{k=0}^{N-1}\ell_k(x_k,u_k)\right] \\
        &\text{s.t.} &&\mathbb{E}_{x_0}^\pi\left[\prod_{k=0}^N\mathbb{1}_{\mathcal{X}^s}(x_k)\right]\geq \alpha.
    \end{align}
\end{subequations}

\textbf{Problem Formulation: Joint Chance Constrained Safe-Optimal Control.}
To eliminate the incentive of the JCC MDP to gain a cost advantage by failing on safety, we generalize \eqref{eq_jcc_in_words} by introducing a measurable function $a:\mathcal{H}_N\rightarrow [0,1]$ in the cost objective and the constraint, 
\begin{subequations}
    \label{eq_main_problem_formulation_using_expectation}
    \begin{align}
        &\inf_{\pi,a} &&\hspace*{-.8em}\mathbb{E}_{x_0}^{\pi}\hspace*{-.1em}\!\left[a(h_N)\!\left(\!\ell_N(x_N)\!+\!\sum_{k=0}^{N-1}\!\ell_k(x_k,u_k)\!\right)\!\right] \label{eq_main_problem_formulation_using_expectation_objective} \\
        &\text{s.t.} &&\hspace*{-.8em}\mathbb{E}_{x_0}^{\pi}\hspace*{-.1em}\!\left[a(h_N)\!\prod_{k=0}^N\!\mathbb{1}_{\mathcal{X}^s}(x_k)\!\right]\!\geq\!\alpha, \label{eq_main_problem_formulation_using_expectation_constraint}
    \end{align}
\end{subequations}
where $\pi\in\Pi$. The function $a$ assigns a weight between zero and one to each trajectory $h_N\in\mathcal{H}_N$. 
Choosing $a(\cdot)=1$ recovers the JCC Problem \eqref{eq_jcc_in_words}. However, we can also select $a(\cdot)$ so that it assigns a value of zero to every unsafe trajectory. This leaves $\pi$ in \eqref{eq_main_problem_formulation_using_expectation} to minimize the cost of safe trajectories only. Figure \ref{fig_illustration_of_jcc_vs_ours} illustrates the conceptual difference between \eqref{eq_jcc_in_words} and \eqref{eq_main_problem_formulation_using_expectation} on a simple MDP with one time-step and two inputs. We numerically demonstrate in Section \ref{sec_num} that the function $a$ prevents JCC policies from violating safety constraints on purpose to reap the cost benefits.

However, solving \eqref{eq_main_problem_formulation_using_expectation} is non-trivial: The function $a$ and product in the constraint render the problem non-Markov. While \eqref{eq_jcc_in_words} can be cast as a constrained MDP on an augmented state space and then solved using DP \cite{schmid2023computing}, it is not obvious whether \eqref{eq_main_problem_formulation_using_expectation} permits a similar transformation. We confirm that this is the case in Section \ref{sec_methods}.

\textbf{Constrained Markov Decision Process.} A constrained MDP over a finite time horizon $N$ is a tuple $\mathcal{M}=(\mathcal{X}, \mathcal{U}, \mathcal{Q}, \ell_{0:N}, g_{0:N},\alpha)$, where $g_k\!:\!\mathcal{X}\!\times\mathcal{U}\!\rightarrow\![0,\infty), k\in[N-1]$ and $g_N\!:\!\mathcal{X}\!\rightarrow\![0,\infty)$ define an additional objective which is constrained by $\alpha$. Given an initial state $x_0\in\mathcal{X}$, the constrained MDP solves 
\begin{subequations}
    \label{eq_standard_constrained_mdp}
    \begin{align}
        &\inf_{\pi\in\Pi} &&\mathbb{E}_{x_0}^{\pi} \left[\ell_{N}(x_N)
        +\sum_{k=0}^{N-1}\ell_k(x_k,u_k)\right] \\
        &\text{s.t.} &&\mathbb{E}_{x_0}^{\pi}\left[g_{N}(x_{N})+\sum_{k=0}^{N-1}g_k(x_k,u_k)\right]\geq \alpha. 
    \end{align}
\end{subequations}
Notably, constrained MDPs are solvable using DP \cite{altman2021constrained}. 


\usetikzlibrary{calc, fit} 
\begin{figure}
    \centering
    \begin{tikzpicture}[
    node distance=1.5cm,
    /utils/exec={\definecolor{safegreen}{HTML}{009E73}\definecolor{unsafeorange}{HTML}{D55E00}\definecolor{safeblue}{HTML}{0072B2}},
    mynode/.style={
        circle,
        draw=safegreen, 
        fill=safegreen!10,
        inner sep=2pt,
        minimum size=6mm,
    },
    mynode_unsafe/.style={
        circle,
        draw=unsafeorange, 
        fill=unsafeorange!10,
        inner sep=2pt,
        minimum size=6mm,
    },
    mylabel/.style={
        font=\small,
        midway,
        auto,
        inner sep=1pt,
    },
    myfit/.style={
        rectangle,
        draw=safeblue, 
        dashed,
        inner sep=3mm,
        label={[font=\normalsize, safeblue, xshift=.75cm, yshift=-0.5cm, anchor=south west]90:Safe Set $\mathcal{X}^s$},
    },
    branch/.style={
        circle,
        fill=safegreen,
        inner sep=0pt,
        minimum size=4pt
    },
    prob/.style={
        font=\scriptsize,
        text=safegreen
    },
    probr/.style={
        font=\scriptsize,
        text=unsafeorange
    }
]

    \node[mynode] (node4) {$0$};
    \node[mynode, above=of node4] (node2) {$2$};
    \node[mynode_unsafe, left=of node2] (node3) {$1$};

    \draw[-stealth, unsafeorange, thick] (node4) -- (node3) 
        node[mylabel, left, xshift=-.2cm, text=black, pos=0.2] {$u_0=0$}
        node[mylabel, left, xshift=-.7cm, yshift=.4cm, text=black, pos=0.2] {$\ell_0(x_0,u_0)=0$};
    
    \draw[-stealth, safegreen, thick] (node4) -- (node2) 
        node[mylabel, right, text=black, xshift=.2cm, pos=0.2] {$u_0=1$}
        coordinate[midway] (midR)
        node[prob, right, pos=0.75] {0.5}
        node[mylabel, right, text=black, xshift=.2cm, yshift=.4cm, pos=0.2] {$\ell_0(x_0,u_0)=1$};

    \node[myfit, fit=(node4) (node2)] {};

    \node[branch] at (midR) {};

    \draw[->, unsafeorange, thick] (midR) to[out=90, in=0] 
        node[probr, left, pos=0.17] {0.5} 
        (node3.east);
\end{tikzpicture}
\caption{Illustration of an MDP with time-horizon $N=1$, state space $\mathcal{X}=\{0,1,2\}$, input space $\mathcal{U}=\{0,1\}$, dynamics $\mathcal{Q}(1|0,0)=\mathcal{Q}(1|1,\cdot)=\mathcal{Q}(2|2,\cdot)=1$, $\mathcal{Q}(1|0,1)=\mathcal{Q}(2|0,1)=0.5$, cost functions $\ell_0(\cdot,0)=0$, $\ell_0(\cdot,1)=1$, $\ell_1(\cdot)=0$, safe set $\mathcal{X}^s=\{0,2\}$, and initial condition $x_0=0$. The input $u_0=0$ results in an unsafe but zero-cost trajectory, while $u_0=1$ generates a cost of $1$ but yields a transition to $x_1=1$ and $x_1=2$ with probability $0.5$ each, resulting in a safe trajectory with probability $\SI{50}{\percent}$. Let $\alpha=0.25$. For Problem \eqref{eq_jcc_in_words}, choosing a policy $\pi=\{\pi_0\}$ described by $\pi_0(0|0)=\pi_0(1|0)=0.5$ yields $\SI{25}{\percent}$ safety and is uniquely optimal. In contrast, for \eqref{eq_main_problem_formulation_using_expectation}, choosing a policy with $\pi_0(1|0)=1$ and $a((x_1,u_0,x_0))=\frac{1}{2}\mathbb{1}_{\{2\}}(x_1)$ is equivalently optimal, leading to $\SI{50}{\percent}$ safety. In both cases, the cost incurred by safe trajectories is always one.}
\label{fig_illustration_of_jcc_vs_ours}
\end{figure}
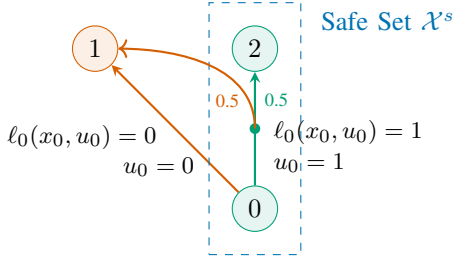


\section{Methodology}
\label{sec_methods}
Problem \eqref{eq_main_problem_formulation_using_expectation} is non-Markov because the product over set-indicator functions in the constraint as well as the function $a$ depend on the full trajectory $h_N$. Consequently, algorithms like DP are not readily applicable. However, as we will show below, it is possible to cast \eqref{eq_main_problem_formulation_using_expectation} as a constrained MDP
\begin{align}
    \mathcal{M}^z=(\mathcal{Z},\mathcal{U}_{0:N},\mathcal{Q}^z_{0:N},f_{0:N+1}, g_{0:N+1}, \alpha)
\end{align}
amenable to DP. The state space of $\mathcal{M}^z$ is $\mathcal{Z}=\mathcal{X}\times\mathcal{C}$ with $\mathcal{C}=[0,\infty)$. We further denote by $\mathcal{Z}^s=\mathcal{X}^s\times \mathcal{C}$. The input space and transition kernel are time-varying: For $k\in[N-1]$, $\mathcal{U}_k=\mathcal{U}$, while $\mathcal{U}_N=\{0,1\}$. For $k\in[N-1]$ $\mathcal{Q}^z_k$ is
\begin{align*}
    \mathcal{Q}^z_k(B\!\times\!C|z_k,u_k) &= \mathcal{Q}(B|x_k,u_k)\delta_{c_k+\ell_k(x_k,u_k)}(C)
\end{align*}
for all $B\subseteq\mathcal{B}(\mathcal{X})$, $C\in\mathcal{B}(\mathcal{C})$, and $z_k=(x_k,c_k)\in\mathcal{Z}$, whereas $\mathcal{Q}^z_N(z|z,\cdot)=1$ for all $z\in\mathcal{Z}$. For $k\in[N-1]$, the cost functions $f_k(\cdot,\cdot)=f_{N+1}(\cdot)=0$, while $f_N(z_N,u_N)=u_N(c_N+\ell_N(x_N))$. The constraint functions are $g_k(\cdot,\cdot)=g_{N+1}(\cdot)=0$ for $k\in[N-1]$, and $g_N(z_N,u_N)=u_N\mathbb{1}_{\mathcal{Z}^s}(z_N)$. Given $z_0=(x_0,0)$, this leads to the optimization problem
\begin{subequations}
    \label{eq_constrained_mdp_formulation}
    \begin{align}
        &\inf_{\pi\in\Pi} &&\mathbb{E}_{z_0}^{\pi} \left[
        f_N(z_N,u_N)\right] \\
        &\text{s.t.} &&\mathbb{E}_{z_0}^{\pi}\left[g_N(z_N,u_N)\right]\geq \alpha.
        \label{eq_constrained_mdp_formulation_constraint}
    \end{align}
\end{subequations}
Any Markov policy $\pi^{\star}$ minimizing \eqref{eq_constrained_mdp_formulation} can be mapped into an optimal solution $(\pi=(\pi_0,\dots,\pi_{N-1}),a)$ of \eqref{eq_main_problem_formulation_using_expectation} via  
\begin{subequations}
    \begin{align}
        \pi_0(\cdot|h_0)&=\pi_0^{\star}(\cdot|(x_0,0)) \\ 
        \pi_k(\cdot|h_k)&=\!\pi^{\star}_k\left(\cdot\middle|\left(x_k,\sum_{j=0}^{k-1}\ell_j(x_j,u_j)\right)\right) \\
        a(h_N)&=\!\pi^{\star}_N\left(1\middle|\left(x_N,\sum_{j=0}^{N-1}\ell_j(x_j,u_j)\right)\right),
    \end{align}
    \label{eq_extracting_optimal_policy_from_cmdp}
\end{subequations}
for all $h_0\in\mathcal{H}_0$, $h_k\in\mathcal{H}_k$, $k=1\dots,N-1$, $h_N\in\mathcal{H}_N$.

\subsection{Derivation of Constrained MDP Formulation}
\label{sec_methods_cmpd}
To make the connection between \eqref{eq_main_problem_formulation_using_expectation} and \eqref{eq_constrained_mdp_formulation} explicit, we start our derivation from the Lagrange dual of \eqref{eq_main_problem_formulation_using_expectation},   
\begin{subequations}%
    \begin{align}%
        \sup_{\lambda\geq 0}\inf_{\pi,a}\mathbb{E}_{x_0}^{\pi}\Bigg[&a(h_N)\Bigg(\ell_N(x_N)\!+\!\sum_{k=0}^{N-1}\!\ell_k(x_k,u_k)\\ &+\lambda\left(\alpha-\prod_{k=0}^N\!\mathbb{1}_{\mathcal{X}^s}(x_k)\right)\Bigg)\!\Bigg].
    \end{align}
    \label{eq_main_problem_formulation_using_expectation_lagrange_dual}%
\end{subequations}
Assuming strong duality holds we can solve \eqref{eq_main_problem_formulation_using_expectation} through \eqref{eq_main_problem_formulation_using_expectation_lagrange_dual}. We could optimize over $\lambda$ via a line-search and attempt solving the inner minimization using DP. For the latter, we assign all terms that depend on $x_N$ to the terminal cost to maintain causality, leading to the value iteration
\begin{subequations}
\label{eq_dp_recursion_lagrange_dual}
\begin{align}
    V_N(x_N) &= \inf_{a} a(h_N)\Bigg(\ell_N(x_N)+\!\sum_{k=0}^{N-1}\!\ell_k(x_k,u_k) \label{eq_dp_recursion_lagrange_dual_terminal}\\ &\qquad \!+\!\lambda\left(\alpha-\prod_{k=0}^N\!\mathbb{1}_{\mathcal{X}^s}(x_k)\right)\Bigg), \nonumber
    \\
    V_k(x_k) &= \inf_{u_k\in\mathcal{U}}\int_{\mathcal{X}}V_{k+1}(x_{k+1})Q(dx_{k+1}|x_k,u_k),
\end{align}
\end{subequations}
with $k\in[N-1]$, $x_k,x_N\in\mathcal{X}$, and $a:\mathcal{H}_N\rightarrow[0,1]$. Unfortunately, the terminal cost now depends on the whole history $h_N$ instead of just $x_N$, again confirming that \eqref{eq_main_problem_formulation_using_expectation} is not Markov. Memorizing $h_N$ up to time $N$ is generally intractable for DP due to the curse of dimensionality.

However, $V_N$ can be evaluated with much less information: The product $\prod_{k=0}^N\mathbb{1}_{\mathcal{X}^s}(x_k)$ equates to one if the history is safe, and zero otherwise, allowing its value to be stored by a binary variable. Further, given a fixed value of $\lambda$ and irrespective of $\pi$, the infimum in \eqref{eq_dp_recursion_lagrange_dual_terminal} is attained by 
\begin{align}
    a(h_N) = \begin{cases}
        0 & \text{if } \ell_N(x_N)\!+\!\sum_{k=0}^{N-1}\!\ell_k(x_k,u_k)\\ &\quad\geq\lambda\left(\prod_{k=0}^N\!\mathbb{1}_{\mathcal{X}^s}(x_k)\!-\!\alpha\right), \\
        1 & \text{otherwise.}
    \end{cases}
    \label{eq_optimal_choice_of_a}
\end{align}
Notably $a(\cdot)=0$ for any unsafe trajectory, but more generally, if the safety and cost of the history $h_N$ would be known at time-step $N$, the terminal value function $V_N$ could be evaluated without knowledge of the history $h_N$ itself. We exploit this fact, and add artificial states to the MDP that capture this information, leading to our construction of $\mathcal{M}^z$ as follows:
\begin{enumerate}[
    label=\textbf{(\roman*)}, 
    leftmargin=*,            
    widest=iii,              
]
    \item \textbf{Terminal state indicates safety.} We assume that the state space is of the form $\mathcal{X}=\mathcal{X}^s\cup x^{\dagger}$. The unsafe set is represented by an absorbing state $x^{\dagger}$ satisfying $\mathcal{Q}(x^{\dagger}|x^{\dagger},u)=1$ for all $u\in\mathcal{U}$. Because $x^{\dagger}$ is absorbing, the product of the safe set indicator functions reduces to $\prod_{k=0}^{N}\mathbb{1}_{\mathcal{X}^s}(x_k)=\mathbb{1}_{\mathcal{X}^s}(x_N)$, which only depends on the terminal state $x_N$. This step can be done without loss of generality: If the unsafe set is not absorbing, one can append a binary variable to the state to keep track of whether the unsafe set has been visited, allowing a similar terminal constraint \cite{schmid2023computing}. 

    \item \textbf{Terminal state indicates cost.} We augment the state space with a cost state $c_k\in\mathcal{C}=[0,\infty)$. The cost state evolves deterministically according to $c_{k+1}=c_k+\ell_k(x_k,u_k)$ with $c_0=0$. The augmented state is $z_k=(x_k,c_k)\in\mathcal{Z}=\mathcal{X}\times\mathcal{C}$. Then, for a given trajectory $(z_0=(x_0,c_0),\dots,z_N=(x_N,c_N))$, the total cost $\ell_N(x_N)+\sum_{k=0}^{N-1}\ell_k(x_k,u_k)=\ell_N(x_N)+c_N$ is fully characterized by the terminal state $z_N$. 

    \item \textbf{Terminal time trajectory weighting.} This reduces $a$ to a function of $z_N=(x_N,c_N)$, where, by abuse of notation,
    \begin{align}
        a(z_N) = \begin{cases}
            0 & \text{if } \ell_N(x_N)\!+\!c_N\!\geq\!\lambda(\mathbb{1}_{\mathcal{X}^s}(x_N)\!-\!\alpha) \\
            1 & \text{otherwise.}
        \end{cases}
    \end{align}
   We embed the function $a$ within the MDP via an additional, virtual time-step $N+1$ and inputs $u_{N}\in\{0,1\}$. 
\end{enumerate}

Under this construction, formalized by $\mathcal{M}^z$, Problem \eqref{eq_main_problem_formulation_using_expectation} reduces to \eqref{eq_constrained_mdp_formulation}. Before proving this formally, we first introduce the following assumption to ensure that a solution to \eqref{eq_constrained_mdp_formulation} exists.
\begin{ass} \ 
    \begin{enumerate}
        \item Problem \eqref{eq_main_problem_formulation_using_expectation} is feasible.
        \item The action space $\mathcal{U}$ is compact, the safe set $\mathcal{X}^s$ is closed.
        \item The cost functions $\ell_{0},\mydots,\ell_N$ are non-negative and continuous over $\mathcal{X}$ and $\mathcal{U}$. 
        \item The kernel $\mathcal{Q}$ is weakly continuous, i.e., $\int_{\mathcal{X}}v(y)\mathcal{Q}(dy|\cdot,\cdot)\in C_b(\mathcal{X}\times\mathcal{U})$ for any $v\in C_b(\mathcal{X})$. 
    \end{enumerate}
    \label{ass_assumptions}
\end{ass}
The first assumption can be verified by computing the maximum safe policy \cite{abate2008probabilistic} and choosing $a(\cdot)=1$. 
\begin{thm}[Strong Duality]
    Under Assumption \ref{ass_assumptions}, the infimum in \eqref{eq_constrained_mdp_formulation} is attained by a Markov policy $\pi^{\star}\!\in\Pi$. Any policy $\pi^{\star}$ that attains the infimum can be mapped into an optimal solution $(\pi,a)$ for \eqref{eq_main_problem_formulation_using_expectation} via \eqref{eq_extracting_optimal_policy_from_cmdp}. \label{thm_strong_duality_main}
\end{thm}
The proof is relegated to Appendix \ref{sec_appendix_proof_of_strong_duality_main}.
While there exists rich literature on solving constrained MDPs \cite{altman2021constrained}, we only describe a DP based approach in detail. We then compare DP against RL as a state-of-the-art alternative in numerical experiments in Section \ref{sec_num}. 

\subsection{Approximate Policy Computation}
\label{sec_discretization}
If $\mathcal{M}$ (and consequently $\mathcal{M}^z$) involves continuous state and action spaces, the DP problem becomes infinite-dimensional. To reduce the state space to finitely many states, we approximate the MDP $\mathcal{M}^z$ via a gridding abstraction. To obtain approximation guarantees we impose the following assumptions in addition to Assumption \ref{ass_assumptions}.
\begin{ass} \
     \begin{enumerate}
        \item The state space $\mathcal{X}$ is compact 
        \item The transition kernel $\mathcal{Q}$ admits a Lipschitz continuous density $q$, such that for some $h_x>0$ \begin{align*}
            |q(x_{k+1}|x_k,u_k) - q(x_{k+1}|x_k',u_k)|\leq h_x||x_k-x_k'||
        \end{align*} for all $x_{k+1},x_k,x_k'\in\mathcal{X}, u_k\in \mathcal{U}$.
     \end{enumerate}    
    \label{ass_lipschitz_cont}
\end{ass}
The compactness of $\mathcal{X}$ and continuity of the cost functions $\ell_{0:N}$ imply that the maximum accumulated cost of any trajectory is bounded by $C_{\text{max}}=\sum_{k=0}^N\|\ell_k\|_{\infty}$, leading to the compact cost space $\mathcal{C}=[0,C_{\text{max}}]$. For completeness, we redefine 
\begin{center}
    {\small $\displaystyle 
    \mathcal{Q}^z_k(B\!\times\!C|z_k,u_k) = \mathcal{Q}(B|x_k,u_k)\delta_{\min(C_{\text{max}},c_k+\ell_k(x_k,u_k))}(C).
    $}
\end{center}
Note that no $N$-step trajectory ever incurs higher costs than $C_{\text{max}}$ by definition, and Theorem \ref{thm_strong_duality_main} remains valid under this construction, see Appendix \ref{sec_appendix_comment}. 

\textbf{Gridding Abstraction.} Consider a finite disjoint partitioning of the cost, state and action spaces 
\begin{align*}
    \mathcal{X}^s\!&=\!\cup_{j=1}^{M_x}{\mathcal{X}}_j,\!&\!\mathcal{X}_j\!&\in\!\mathcal{B}(\mathcal{X}),\!&\!\mathcal{X}^r&\!=\!\{{{x}}^r_j\!\in\!{\mathcal{X}}_j\}_{j=1}^{M_x}\!\cup\!x^{\dagger}, \\
    \mathcal{C}\!&=\!\cup_{j=1}^{M_c}\mathcal{C}_j,\!&\!\mathcal{C}_j\!&\in\!\mathcal{B}(\mathcal{C}),\!&\!\mathcal{C}^r&\!=\!\{{{c}}^r_j\!\in\!\mathcal{C}_j\}_{j=1}^{{M_c}},\! \\
    \mathcal{U}\!&=\!\cup_{j=1}^{M_u}\mathcal{U}_j,\!&\!\mathcal{U}_j\!&\in\!\mathcal{B}(\mathcal{U}),\!&\!\mathcal{U}^r&\!=\!\{{{u}}^r_j\!\in\!\mathcal{U}_j\}_{j=1}^{M_u},
\end{align*}
where $\mathcal{X}^r, \mathcal{C}^r, \mathcal{U}^r$ are collections of representative points of each partition. As in Section \ref{sec_methods_cmpd} we construct the product of the cost and state spaces to cover $\mathcal{Z}=\left(\bigcup_{i=1}^{M_x}\bigcup_{l=1}^{M_c}\mathcal{X}_i\times \mathcal{C}_l\right)\cup\left(\bigcup_{l=1}^{M_c}\{x^{\dagger}\}\times \mathcal{C}_l\right)$. We denote each of the $M_z=(M_x+1)M_c$ products by $\mathcal{Z}_j$ with associated state $z_j^r\in\mathcal{Z}_j$, $j=1,\dots,M_z$, and $\mathcal{Z}^r = \{z_j^r\}_{j=1}^{M_z}$. Of course, beyond continuous state spaces, the same construction allows reducing large finite state spaces to smaller ones by grouping states together.

We now construct a finite MDP 
\begin{align*}
    {\mathcal{M}}^{rz}=({\mathcal{Z}}^r, {\mathcal{U}}^r_{0:N}, \mathcal{Q}^{rz}_{0:N}, f_{0:N+1}, g_{0:N+1}, \alpha),
\end{align*}
where $\mathcal{Q}^{rz}_k(z^r_i|z,{u})=\mathcal{Q}^z_k(\mathcal{Z}_i|{z},u)$ for all $i=1,\dots,M_z$, $z\in\mathcal{Z}$, $u\in\mathcal{U}$, $k\in[N]$, and $\mathcal{U}_k^r=\mathcal{U}^r$ for $k\in[N-1]$ and $\mathcal{U}^r_{N}=\mathcal{U}_N$. The interpretation is that from any state within a partition $\mathcal{Z}_j$ the transition probabilities are the same as from the state $z^r_j$, the inputs are restricted to the finitely many representative inputs, and instead of the actual stage cost one incurs the stage cost of the associated partition's representative states and actions.

\textbf{Policy Computation.}
We solve Problem \eqref{eq_constrained_mdp_formulation} on the MDP $\mathcal{M}^{rz}$ through its Lagrange dual
\begin{align}
\label{eq_constrained_mdp_formulation_lagrangian}
    &\max_{\lambda\geq 0}\min_{\pi\in\Pi} \ \mathbb{E}_{z_0}^{\pi} \Bigg[f_{N}(z_{N},u_N)\!+\!\lambda\left(\alpha\!-\!g_{N}(z_{N},u_N)\right)\Bigg],
\end{align}
where $\lambda$ is the dual multiplier. The inner minimization is solved via DP, while the outer maximization is performed through bisection on $\lambda$. Constructing mixtures of policies obtained at the upper and lower bound of the bisection ensures exponential convergence to the optimal solution over the number of bisection steps, see \cite{schmid2023computing} for details. 

The policy is executed on the original MDP ${\mathcal{M}}$ by applying the action $u_k\sim\pi_k(\cdot|z_j^r)$ when $z_k\in\mathcal{Z}_j$. The cost state $c_k$ is propagated virtually in the controller software and, for technical reasons detailed in Appendix \ref{sec_appendix_proof_of_error_bounds}, we always increase the cost state $c_k$ by the cost of the associated representative states instead of the actually incurred costs.

\textbf{Error Bound}. Note that the gridding abstraction only builds an approximation of the original system. Hence, when solving for a safety of $\alpha$ on the discretized MDP $\mathcal{M}^{rz}$ and applying the policy on the continuous MDP $\mathcal{M}$ one might expect the probability of safety to differ from $\alpha$. In what follows we bound this difference. 

We define the state-grid-size 
\begin{align}
    \Delta_x = \max_{1\leq j \leq M_x}\sup_{{x}_1,{x}_2\in {\mathcal{X}}_j}||{x}_1-{x}_2||,
\end{align}
and the maximum Lebesgue measure of the sets $\mathcal{X}_j$, $j=1,\mydots,M_x$ as 
\begin{align}
    \overline{\mu}=\max_{1\leq j \leq M_x}\int_{\mathcal{X}_j}1dx.
\end{align}
\begin{thm}[Safety Error Bounds]\label{prop_error_bounds}
     Under Assumption \ref{ass_assumptions} and \ref{ass_lipschitz_cont}, a feasible policy of Problem \eqref{eq_constrained_mdp_formulation} yields at least a safety of $\alpha -  N h_xM_x\overline{\mu}\Delta_x$ when applied to the original MDP $\mathcal{M}$.
\end{thm}
\smallskip
The proof is found in Appendix \ref{sec_appendix_proof_of_error_bounds} and is based on the analysis in \cite{abate2010approximate}. A safety of $\alpha$ on $\mathcal{M}$ can be consequently achieved by tightening \eqref{eq_constrained_mdp_formulation_constraint} to a safety of $\alpha + N h_xM_x\overline{\mu}\Delta_x$ on $\mathcal{M}^{rz}$. This requires that $Nh_xM_x\overline{\mu}\Delta_x<1-\alpha$, which can be ensured by refining the partition. For example, if we impose a regular grid structure $\overline{\mu}M_x$ is roughly independent of the grid size $\Delta_x$ and equal to the Lebesgue measure of the compact set $\mathcal{X}^s$, so the term decreases proportionately to $\Delta_x$ ($h_x$ and $N$ are of course independent of the grid size). Note, however, that as the grid size decreases, the number of partitions (and hence the necessary computation) increases exponentially in the dimension of the continuous space $\mathcal{X}$.

\begin{figure*}[!t]
    \centering
    \begin{subfigure}{0.24\linewidth}
        \centering
        \includegraphics[width=\linewidth]{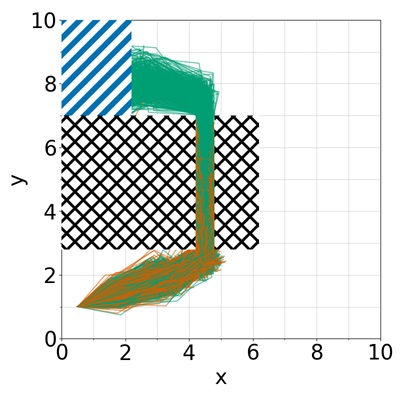}
    \end{subfigure}
    \begin{subfigure}{0.24\linewidth}
        \centering
        \includegraphics[width=\linewidth]{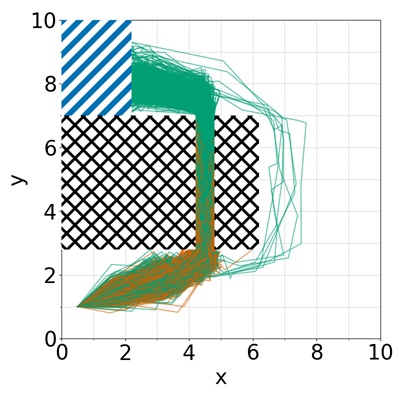}
    \end{subfigure}
    \begin{subfigure}{0.24\linewidth}
        \centering
        \includegraphics[width=\linewidth]{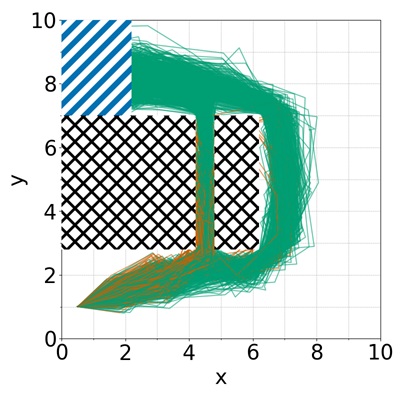}
    \end{subfigure}
    \begin{subfigure}{0.24\linewidth}
        \centering
        \includegraphics[width=\linewidth]{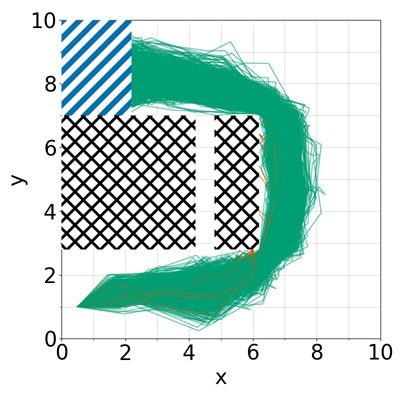}
    \end{subfigure}
    \caption{Comparison of $1000$ Monte-Carlo simulations using policies computed for \eqref{eq_constrained_mdp_formulation_lagrangian} via DP with $\lambda\in\{14,16,18,20\}$ from left to right. The target is blue striped, the unsafe set black hashed. Green trajectories satisfy the reach-avoid specification, red trajectories do not. The initial state is $(0.5,1)$.}
    \label{fig_dp_trajectories_different_lambdas}
\end{figure*}

\subsection{Resolving Ambiguity}
\label{sec_method_ambiguity}
We return to the interpretation of $a$ in \eqref{eq_main_problem_formulation_using_expectation} as (de-)selecting individual trajectories. If $a(\cdot)=0$, the respective trajectory does not affect the cost, nor the constraint objective. Note that the choice of $a$ in \eqref{eq_optimal_choice_of_a} also maps safe trajectories to zero whenever they exceed a certain cost. This is intentional: Strictly setting $a$ to zero for all unsafe trajectories and one otherwise incentivises policies to reduce objective \eqref{eq_main_problem_formulation_using_expectation_objective} by making high-cost trajectories unsafe, potentially at the cost of worsening performance for safe trajectories. In contrast, if $a$ in \eqref{eq_optimal_choice_of_a} is zero, then the cost and constraint objective of \eqref{eq_main_problem_formulation_using_expectation} are oblivious to the respective trajectories' cost and safety. Hence, one can choose $\pi$ to keep any such trajectory safe without impacting \eqref{eq_main_problem_formulation_using_expectation}. 

This implies ambiguity in the solution of \eqref{eq_main_problem_formulation_using_expectation} where policies with different levels of safety qualify as optimal, see Fig.~\ref{fig_illustration_of_jcc_vs_ours}. However, it is possible to retrieve the safest policy from the set of optimal solutions for \eqref{eq_main_problem_formulation_using_expectation}: Let $c^{\star}$ denote the optimal objective in \eqref{eq_constrained_mdp_formulation} and solve 
\begin{subequations}
    \begin{align}
        &\sup_{\pi\in\Pi} &&\mathbb{E}_{z_0}^{\pi}\left[g_{N}(z_{N},u_N)\right] \\
        &\text{s.t.} &&\mathbb{E}_{z_0}^{\pi} \left[f_{N}(z_{N},u_N)\right]=c^{\star}.
    \end{align}
    \label{eq_constrained_mdp_formulation_refining}
\end{subequations}
\begin{prop}[Optimal Markov Policies]
    Under Assumption \ref{ass_assumptions}, the supremum in \eqref{eq_constrained_mdp_formulation_refining} is attained by a Markov policy.    \label{prop_constrained_mdp_formulation_refining}
\end{prop}
See Appendix \ref{sec_appendix_attainability_refined} for the proof.

\section{Numerical Experiments}
\label{sec_num}
We evaluate our approach on a unicycle system modeled by a two-dimensional state space $(x_{1,k},x_{2,k})\in[0,10]\times[0,10]$, two inputs $(u_{1,k},u_{2,k})\in[0,3]\times[0,2\pi]$, representing the velocity and heading angle, and dynamics 
\begin{align*}
    \begingroup
    \setlength\arraycolsep{3pt}x_{k+1}\!=\!x_k\!+\!u_{1,k}\begin{bmatrix}\cos(u_{2,k}\!+\!w_{2,k}) & \sin(u_{2,k}\!+\!w_{2,k})\end{bmatrix}^{\top}\!+\!w_{1,k},
    \endgroup
\end{align*}
where $w_{1,k}\sim\mathcal{N}([0 \ 0]^{\top},\text{diag}(0.025,0.025))$ is an additive disturbance and $w_{2,k}\sim\mathcal{N}(0,0.01)$ steering actuation noise. The objective is to safely reach targets in cluttered environments with as little velocity actuation $\sum_{k=0}^{N-1}u_{1,k}$ as possible, but with probability at least $\alpha$ over $N=15$ time-steps. While this is a standard stochastic control problem, we invite the reader to think about the expected optimal behaviour of the controller. 

Generally, one would expect a risk-aware control behaviour in which trajectories potentially graze unsafe sets if this allows reaching the target with little actuation. The smaller $\alpha$, the riskier these manoeuvres as violations are allowed to occur more likely. We will demonstrate that our formulation \eqref{eq_main_problem_formulation_using_expectation} matches this intuition, while standard JCC optimal control problems \eqref{eq_jcc_in_words} do not. 

To provide clearer visualizations (avoid safe trajectories cutting corners) we define trajectories as safe if the line between any consecutive states does not intersect the unsafe set. While the previous discussion focused on safety-constraints, we will consider probabilistic reach-avoid constraints with absorbing unsafe and target sets for the numerical examples. Note that reach-avoid specifications with absorbing unsafe and target sets are easily cast as safety specifications by reducing the safe set to the target set at the terminal time. The code used to generate all numerical results has been executed on a Ryzen 9 9950X CPU at $\SI{4.3}{\giga\hertz}$, 32 GB RAM, and an Nvidia RTX 5070 Ti GPU and is openly accessible on \url{https://github.com/NiklasSchmidResearch/JCC_Safe_Optimal}.

\subsection{Dynamic Programming based Experiments}
We first approach Problem \eqref{eq_main_problem_formulation_using_expectation} via DP. To render the problem finite dimensional, we grid the state space into partitions of size $0.2\times 0.2$, and generate $\mathcal{U}^r$ by uniformly sampling a set of $50$ input pairs from $\mathcal{U}$. For each action $u_k\in\mathcal{U}^r$, we exploit the translation- and time-invariance of the dynamics to simulate $i\in[500]$ transitions from the origin $x_k=0$ to states $x_{k+1}=x^i$. In the DP algorithm, we empirically evaluate expectations over the transition probabilities given any $x_k$ and $u_k$ through the samples $(x+x^i, y+y^i)$. Further, the cost space is discretized into partitions of size $0.2$, leading to a combined state space of $50\times 50 \times 225$ states. Solving the inner minimization in \eqref{eq_constrained_mdp_formulation_lagrangian} for a fixed value of $\lambda$ took between $80$ and $110$ seconds.

\begin{figure}
    \centering
    \definecolor{safegreen}{HTML}{009E73}
    
    \begin{tikzpicture}
    \begin{axis}[
        xlabel={Safety},
        ylabel={Expected Cost},
        legend style={
            at={(0.435,1.45)}, 
            anchor=north,
            legend columns=2, 
            font=\small
        },
        grid=major,
        width=\linewidth,
        height=4cm,
    ]

    \addplot[color=safegreen, ultra thick, solid] table [x=Safety, y=SafeCost] {data_pareto_jcc_safe.tex};
    \addlegendentry{Safe-Optimal (Safe Traj.)}

    \addplot[color=black, ultra thick, solid] table [x=Safety, y=AllCost] {data_pareto_jcc_safe.tex};
    \addlegendentry{Safe-Optimal (All Traj.)}

    \addplot[color=safegreen, ultra thick, dashed] table [x=Safety, y=SafeCost] {data_pareto_jcc_classic.tex};
    \addlegendentry{Classic JCC (Safe Traj.)}

    \addplot[color=black, ultra thick, dashed] table [x=Safety, y=AllCost] {data_pareto_jcc_classic.tex};
    \addlegendentry{Classic JCC (All Traj.)}

    \end{axis}
    \end{tikzpicture}
    \caption{Pareto Front: Expected cost of safe / all trajectories against safety for the classical JCC problem \eqref{eq_jcc_in_words} and our safe-optimal approach \eqref{eq_main_problem_formulation_using_expectation} for the slit environment, computed using $1000$ Monte-Carlo simulations of each policy for different values of $\lambda$.}
    \label{fig_pareto_costs}
\end{figure}
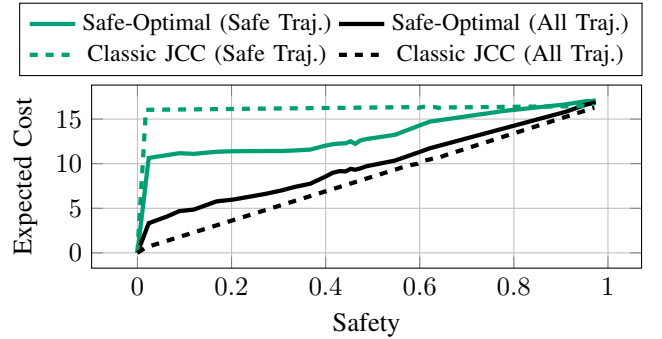

\textbf{Performance and Safety is Balanced via $\bm{\lambda}$.}
We first treat the dual variable $\lambda$ as a parameter and analyse the effect of its value on a reach-avoid problem with a target in the top-left corner and an unsafe set with a slit, see Fig.~\ref{fig_dp_trajectories_different_lambdas}. For simplicity, we call any trajectory that satisfies the reach-avoid objective safe, and any trajectory that enters the unsafe set or does not reach the target within $15$ time-steps unsafe. Generally, two routes to the target are available. The shorter route leads through the thin slit, but has high risk of entering the unsafe set. The second route goes around the unsafe set and is less risky but significantly longer. In \eqref{eq_constrained_mdp_formulation_lagrangian}, a lower value of $\lambda$ implies a lower reward on safety. Indeed, for $\lambda=14$, the policy exploits the slit through the unsafe set as a shortcut to the target. As $\lambda$ is increased to $20$, safety is more strongly rewarded and the policy prefers the longer, but less risky route around the unsafe set. This trend is also summarized in Table \ref{tab_dp_trajectories_different_lambdas}, where the safety and cost of safe trajectories increase with $\lambda$. The value of $\lambda$ therefore allows trading off safety against cost of safe trajectories. 

\begin{figure*}[!t]
    \centering
    
    
    \begin{minipage}[b]{0.03\linewidth}
        \centering
        \rotatebox{90}{\textbf{\small Opt. Cost of Safe Traj.}}
        \vspace{2em} 
    \end{minipage}
    \hspace{-1em} 
    \begin{subfigure}{0.23\linewidth}
        \centering
        \includegraphics[width=\linewidth]{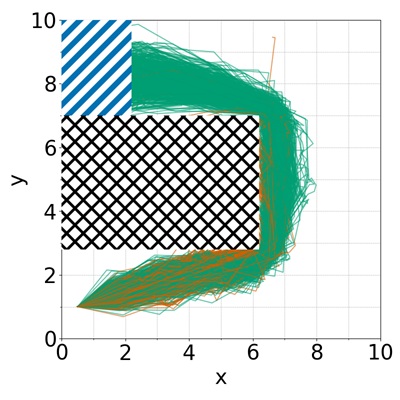}
    \end{subfigure}
    \begin{subfigure}{0.23\linewidth}
        \centering
        \includegraphics[width=\linewidth]{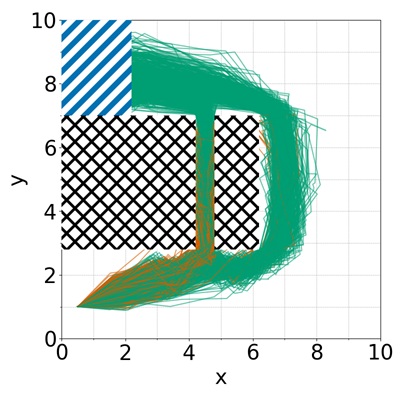}
    \end{subfigure}
    \begin{subfigure}{0.23\linewidth}
        \centering
        \includegraphics[width=\linewidth]{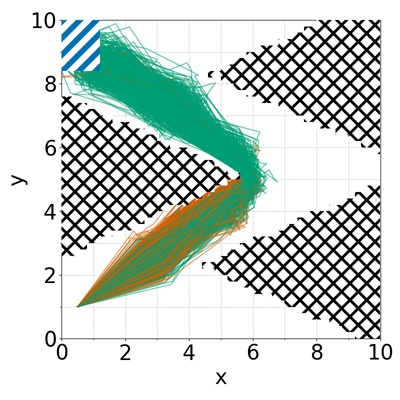}
    \end{subfigure}
    \begin{subfigure}{0.23\linewidth}
        \centering
        \includegraphics[width=\linewidth]{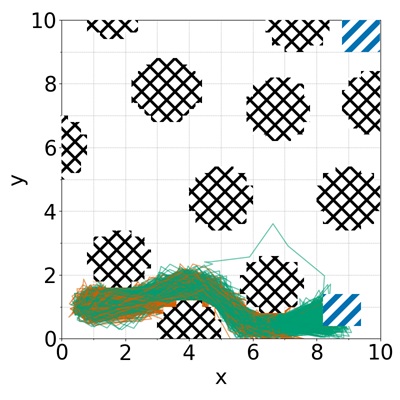}
    \end{subfigure}

    \vspace{0em} 

    \begin{minipage}[b]{0.03\linewidth}
        \centering
        \rotatebox{90}{\textbf{\small Opt. Cost of All Traj.}}
        \vspace{2em} 
    \end{minipage}
    \hspace{-1em} 
    \begin{subfigure}{0.23\linewidth}
        \centering
        \includegraphics[width=\linewidth]{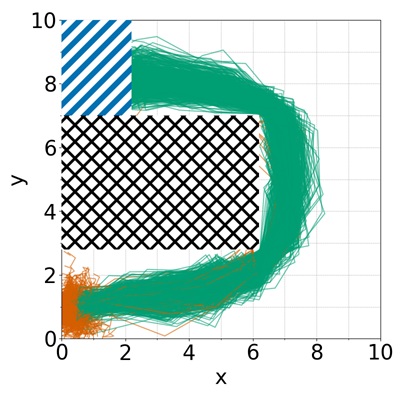}
    \end{subfigure}
    \begin{subfigure}{0.23\linewidth}
        \centering
        \includegraphics[width=\linewidth]{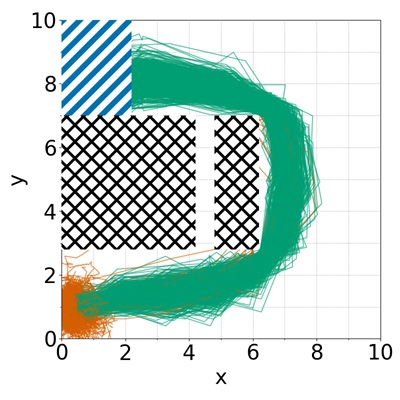}
    \end{subfigure}
    \begin{subfigure}{0.23\linewidth}
        \centering
        \includegraphics[width=\linewidth]{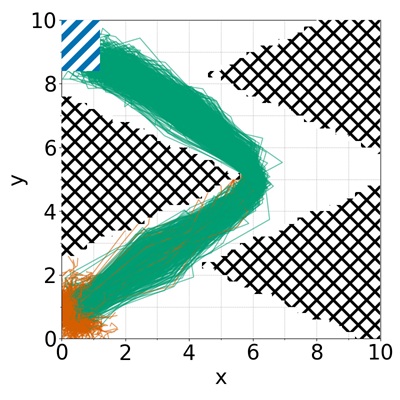}
    \end{subfigure}
    \begin{subfigure}{0.23\linewidth}
        \centering
        \includegraphics[width=\linewidth]{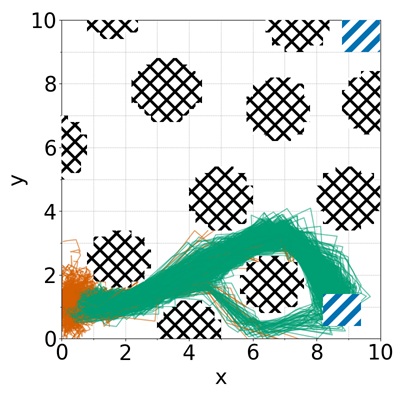}
    \end{subfigure}

    \caption{Comparison of $1000$ Monte-Carlo simulations generated using the optimal policies for problems \eqref{eq_main_problem_formulation_using_expectation} (first row) and \eqref{eq_jcc_in_words} (second row) with $\alpha=0.6$ and initial state $(0.5,1)$ solved via DP. Environments: Simple (left), Simple-Slit (middle-left), Zigzag (middle-right), Balls (right). Target: blue striped, unsafe set: black hashed, safe trajectories: green, unsafe trajectories: red.}
    \label{fig_dp_trajectories_jcc_vs_ours}
\end{figure*}

\begin{table}[!t]
\centering
\begin{tabular}{ccc}
\hline
$\bm{\lambda}$ & \textbf{Safety Probability} & \textbf{Exp. Cost of Safe Trajectories} \\ \hline
14 & 0.46 & 12.41 \\
16 & 0.47 & 12.63 \\
18 & 0.64 & 14.61 \\
20 & 0.96 & 16.99 \\ \hline
\end{tabular}
\caption{Safety vs. expected cost of safe trajectories for trajectories in Fig.~\ref{fig_dp_trajectories_different_lambdas} empirically estimated over $1000$ Monte-Carlo simulations.}
\label{tab_dp_trajectories_different_lambdas}
\end{table}

\textbf{Comparison of Problems \eqref{eq_jcc_in_words} and \eqref{eq_main_problem_formulation_using_expectation}.}
Next, we use a heuristic line-search over $\lambda$ to find policies with safety $\alpha=0.6$ in different reach-avoid environments, see Fig.~\ref{fig_dp_trajectories_jcc_vs_ours}. We compare the results to the standard JCC policies obtained from solving Problem \eqref{eq_jcc_in_words} via the procedure described in \cite{schmid2023computing}. 

Over safe and unsafe trajectories, the policies obtained from \eqref{eq_jcc_in_words} achieve a lower expected cost than the policies obtained from \eqref{eq_main_problem_formulation_using_expectation}, see Table~\ref{tab_cost_comparison}. However, this is mostly due to low-cost trajectories that remain close to the initial state and do not reach the target; those trajectories that reach the target follow more conservative routes and incur higher costs than those of \eqref{eq_main_problem_formulation_using_expectation}. Our approach only generates unsafe trajectories as “collateral damage” during risky manouvers along unsafe sets. Those trajectories that remain safe during such manouvers benefit from low costs. In applications where safety is traded in favor of performance, we argue that the behaviour of the policies obtained from \eqref{eq_main_problem_formulation_using_expectation} is more suitable than the one of those obtained from \eqref{eq_jcc_in_words} as it matches common intuition and the expected outcome of constrained stochastic control problems more closely. The Pareto fronts in Fig.~\ref{fig_pareto_costs} indicate that no performance on safe trajectories is gained by increasing risks in \eqref{eq_jcc_in_words}. The joint cost of safe and unsafe trajectories increases linearly with the safety $\alpha$ while the cost of safe trajectories remains unchanged; increasing the safety probability merely varies the number of trajectories that remain at the initial state.

We did not find any quantifiable ambiguity in the optimal policy for systems with a non-trivial number of time-steps. Any non-zero probability of noise sequences leading to low-cost safe trajectories, even from unfavorable states, poses unique trade-offs on the achievable safety and cost. 



\subsection{Reinforcement Learning based Experiments}
We compare our DP approach to RL as a state-of-the-art alternative for solving constrained MDPs. For ease of comparison, we remain with our 2D unicycle example. 

In contrast to grid-based DP, the bottleneck of RL is not the dimensionality of the system's state space but sparsity in the rewards \cite{bonyadi2022self}. Unfortunately, temporal logic specifications, such as safety or reach-avoid constraints, feature highly sparse rewards: In \eqref{eq_constrained_mdp_formulation_lagrangian} the constraint is enforced through a reward of $\lambda$, which is only achieved at terminal time and only if the system reached the target; the reward is always zero otherwise. This can render learning slow, or even unstable. 

\begin{figure*}[!t]
    \centering
    \begin{subfigure}{0.24\linewidth}
        \centering
        \includegraphics[width=\linewidth]{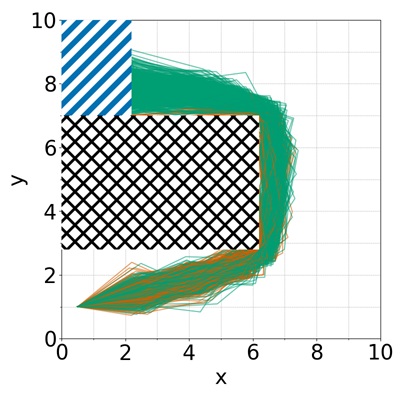}
    \end{subfigure}
    \begin{subfigure}{0.24\linewidth}
        \centering
        \includegraphics[width=\linewidth]{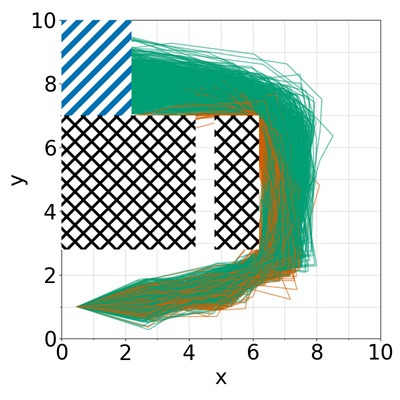}
    \end{subfigure}
    \begin{subfigure}{0.24\linewidth}
        \centering
        \includegraphics[width=\linewidth]{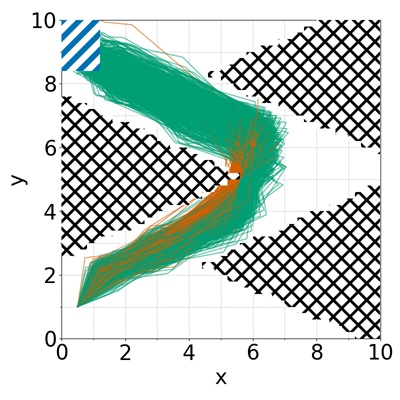}
    \end{subfigure}
    \begin{subfigure}{0.24\linewidth}
        \centering
        \includegraphics[width=\linewidth]{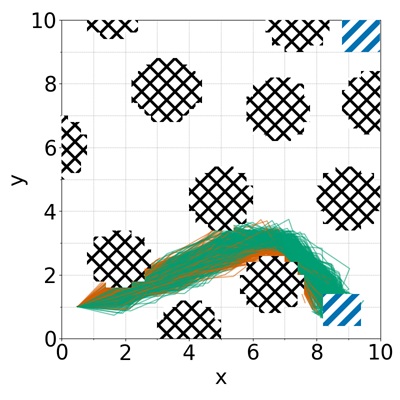}
    \end{subfigure}
    \caption{Trajectories generated by RL-policies solving the safe-optimal Problem \eqref{eq_main_problem_formulation_using_expectation} for the environments in Fig.~\ref{fig_dp_trajectories_jcc_vs_ours}.}
    \label{fig_rl_trajectories_different_worlds_safeties}
\end{figure*}

To overcome this sparse reward structure, we implement a lexicographic learning framework; all steps are implemented via the soft-actor-critic algorithm \cite{haarnoja2018soft} whose implementation we adopted from \cite{khan2024allrl}: 
\begin{enumerate}
    \item We first train the actor and critic to achieve a large reach-avoid probability from arbitrary initial states. To achieve this, we choose a large value of $\lambda$ and uniformly sample initial states across the state space.  
    \item We stop randomizing the initial state but train the actor and critic from the true initial state of the system. If $\lambda$ is chosen high enough, this generally results in a policy that achieves a reach-avoid probability greater than $\alpha$.
    \item The final step involves updating $\lambda$ until the policy achieves a safety of approximately $\alpha$. To achieve this, we apply the update
    \begin{align}
      \lambda^{i+1} &= \lambda^i - \eta_\lambda(\mathbb{E}[g_N(z_{N},u_N)] - \alpha) 
    \end{align}
    after every episode $i$ where the expectation is empirically approximated via $1000$ rollouts. Note that the actor and critic do not depend on $\lambda$. To ensure stable convergence, we choose $\eta_{\lambda}$ approximately equal to the critic learning rate, while the actor learning rate is three times smaller.
\end{enumerate}
Each step is executed until convergence or until a maximum number of episodes has been reached. The actor and critic are realized by neural networks with two fully connected hidden layers and $32$ nodes per layer.


\textbf{Comparison of RL with DP.} We execute the RL approach on the same environments as DP. Fig.~\ref{fig_rl_trajectories_different_worlds_safeties} depicts trajectories obtained from Monte-Carlo simulations of the resulting policies. The average costs of all, only safe and only unsafe trajectories are listed in Table \ref{tab_cost_comparison} and compared to DP. In general, the trajectory rollouts display a similar behaviour for the RL and DP policy, which verifies the general applicability of RL to Problem \eqref{eq_main_problem_formulation_using_expectation}. However, DP outperformed RL in all except the Simple-environment. Particularly in the Simple-Slit environment the RL policy did not exploit the slit to reduce costs which indicates convergence to a locally optimal solution. The conservative routes chosen by RL are potentially due to our lexicographic learning framework, which approaches a cost-safety balancing policy from a maximum safe policy. 

Overall, computing an optimal policy took significantly longer with RL, ranging from $91$ minutes for the Zigzag environment to $132$ minutes for the Simple-Slit environment (Fig.~\ref{fig_rl_trajectories_different_worlds_safeties}, third and second column, respectively). However, we remark that the computational complexity of grid-based DP scales exponentially with its state space dimensionality, which often hampers its applicability to real-world problems. In contrast, RL proved successful in solving high-dimensional, real-world MDPs at which DP becomes infeasible \cite{silver2018general, miki2022learning}. We do not aim to display the general scalability of RL, but rather compare the quality of its solutions against our DP results.

\begin{table}[h]
\centering
\begin{tabular}{@{}llcccc@{}}
\toprule
\multirow{2}{*}{\textbf{Environment}} & \multirow{2}{*}{\textbf{Expected Cost of}} & \textbf{Standard JCC} & \multicolumn{2}{c}{\textbf{Safe-Opt.}} \\ \cmidrule(lr){3-3} \cmidrule(lr){4-5} 
 &  & \textbf{DP} & \textbf{DP} & \textbf{RL} \\ \midrule
\multirow{3}{*}{\textbf{Simple}} 
 & All Trajectories & 9.84 & 12.53 & 12.49 \\
 & Unsafe Trajectories & 0.96 & 6.85 & 7.77 \\
 & Safe Trajectories & 16.38 & 16.23 & 15.30 \\ \midrule
\multirow{3}{*}{\textbf{Simple-Slit}} 
 & All Trajectories & 9.46 & 11.39 & 14.62 \\
 & Unsafe Trajectories & 0.68 & 6.75 & 11.67 \\
 & Safe Trajectories & 16.23 & 14.31 & 16.30 \\ \midrule
\multirow{3}{*}{\textbf{Zigzag}} 
 & All Trajectories & 8.47 & 10.16 & 12.15 \\
 & Unsafe Trajectories & 0.72 & 6.83 & 8.26 \\
 & Safe Trajectories & 13.30 & 12.63 & 14.45 \\ \midrule
\multirow{3}{*}{\textbf{Balls}} 
 & All Trajectories & 5.85 & 5.93 & 7.93 \\
 & Unsafe Trajectories & 0.73 & 4.10 & 5.28 \\
 & Safe Trajectories & 9.21 & 7.04 & 9.31 \\ \bottomrule
\end{tabular}
\caption{Average trajectory costs over $1000$ Monte-Carlo simulations for the standard JCC policy \eqref{eq_jcc_in_words} computed via DP, and the safe-optimal JCC policy \eqref{eq_main_problem_formulation_using_expectation} computed via DP and RL for different environments.}
\label{tab_cost_comparison}
\end{table}

\section{Conclusion \& Outlook}
\label{sec_concl}
We addressed shortcomings of the standard joint chance constraint problem for MDPs by a novel formulation in which only the performance of constraint-satisfying trajectories is optimized. We demonstrated that this problem variant avoids phenomena in which chance constrained controllers actively enforce constraint violations in favor of low cost. We proved that our problem can be cast as a constrained Markov Decision Process on an augmented state space and analysed dynamic programming and reinforcement learning based solution methods. We proposed bounds on the policies' safety when solved on a gridding abstraction of the MDP. Future work will analyse the effect of our formulation on more general temporal logic constraints \cite{haesaert2021formal}, build connections to conditional value-at-risk problems, introduce further problem variants based on our formulation, and validate our approach on real-world systems.

\bibliographystyle{IEEEtran}
\bibliography{references}

\appendix
\label{sec_appendix}
\ifArxiv
\counterwithin{theorem}{subsection}
\fi

The following sections contain the proofs of our main Theorems \ref{thm_strong_duality_main} and \ref{prop_error_bounds}, as well as Proposition \ref{prop_constrained_mdp_formulation_refining}. Throughout, we will integrate over histories $h_k$ while the integrated term is only a function of states and actions contained in $h_k$; we assume that it is clear from the context that the respective states and actions are associated to the integrated history. Further, we denote the set of histories in $\mathcal{M}^{z}$ and $\mathcal{M}^{rz}$ for $k\in[N+1]$ by $\mathcal{H}^{z}_k=\prod_{j=0}^{k-1}(\mathcal{Z}\times\mathcal{U}_j)\times\mathcal{Z}$ and $\mathcal{H}^{rz}_k=\prod_{j=0}^{k-1}(\mathcal{Z}^r\times\mathcal{U}^r_j)\times\mathcal{Z}^r$, respectively.

\subsection{Proof of Theorem \ref{thm_strong_duality_main}}
\label{sec_appendix_proof_of_strong_duality_main}
Let $\psi:\mathcal{H}_k\rightarrow \mathcal{H}^z_k$ be the measurable mapping that associates to any history $h_k\in\mathcal{H}_k$ the unique associated sequence $h^z_k\in\mathcal{H}^z_k$ with the same states and actions in $\mathcal{X}$ and $\mathcal{U}$ for all $k\in[N]$, that is 
\begin{align*}
    \psi((x_k,u_{k-1},\dots,x_0))=((x_k,c_k),u_{k-1},\dots,(x_0,c_0))
\end{align*}
with $c_0=0$, $c_k=c_{k-1}+\ell_{k-1}(x_{k-1},u_{k-1})$ for $k=1,\dots,N$. Note that $\psi$ is bijective, since $h_k$ and $h^z_k$ only differ in the additional cost-states $c_0,\dots, c_k$, which can be uniquely reconstructed from the state-action history.

\begin{defn}
    Let $(a,\pi)$ be feasible for \eqref{eq_main_problem_formulation_using_expectation}. We call $\pi^{z}=\{\pi_0^{z},\dots,\pi_N^{z}\}$ with $\pi_k^{z}(\cdot|h_k^z)=\pi_k(\cdot|\psi^{-1}(h_k^z))$ for all $h_k^z\in\mathcal{H}^z_k$ and $\pi_N^{z}(1|h_N^z)=a(\psi^{-1}(h_N^z))$, $\pi_N^{z}(0|h_N^z)=1-\pi_N^{z}(1|h_N^z)$ an adaptation of $(a,\pi)$ to \eqref{eq_constrained_mdp_formulation}. Vice versa, let $\pi^z$ be feasible for \eqref{eq_constrained_mdp_formulation}. We call $(\pi,a)$ with $\pi=\{\pi_0,\mydots,\pi_{N-1}\}$, $\pi_k(\cdot|h_k)=\pi_k^z(\cdot|\psi(h_k))$ for any $h_k\in\mathcal{H}_k$, $k\in[N-1]$ and $a(\cdot)=\pi_N^z(1|\psi(\cdot))$ an adaptation of $\pi$ to \eqref{eq_main_problem_formulation_using_expectation}.
    \label{def_policy_adaptation}
\end{defn}
By construction, if $(a,\pi)$ is an adaptation of $\pi^z$ to \eqref{eq_main_problem_formulation_using_expectation}, then $\pi^z$ is an adaptation of $(a,\pi)$ to \eqref{eq_constrained_mdp_formulation}. 

For all $k\in[N]$, executing a policy $\pi\in\Pi$ with initial condition $x_0\in\mathcal{X}$ on $\mathcal{M}$ induces an occupation measure on $\mathcal{H}_k$ \cite{hernandez2012discrete}, which we denote by $p_k$. Likewise, we denote occupation measures on $\mathcal{H}_{0:N+1}^z$ by $p_{0:N+1}^z$. The adaptation of policies in Definition \ref{def_policy_adaptation} induces a relation between their associated occupation measures $p_{0:N}$ and $p_{0:N+1}^z$.
\begin{lem}
    Let $(a,\pi)$ be an adaptation of $\pi^z$ to \eqref{eq_main_problem_formulation_using_expectation}, or vice versa to \eqref{eq_constrained_mdp_formulation}. Let $p_{0:N}$ be the occupation measures on $\mathcal{H}_{0:N}$ induced by executing $\pi$ on $\mathcal{M}$. Let $p_{0:N+1}^z$ be the occupation measures on $\mathcal{H}^z_{0:N}$ induced by executing $\pi^z$ on $\mathcal{M}^z$. Then, $p_k^z = \psi_\# p_k$ for all $k\in[N]$, where $\psi_\#$ denotes the push-forward through $\psi$.
    \label{lem_occupation_measure_pushforward}
\end{lem}
\begin{proof}
    We use the induction hypothesis that, up to time-step $k$, $p_k^z = \psi_\# p_k$, and show that this implies $p_{k+1}^z = \psi_\# p_{k+1}$. Indeed, for any $B\in\mathcal{B}(\mathcal{H}_{k+1}^z)$,
    \begin{subequations}
        \begin{align}
            p_{k+1}^z(B)\!&=\!\mathbb{P}_{x_0}^{\pi}(h_{k+1}^z\in B|h_k^z)
            \\
            &=\!\int_{\mathcal{H}_k^z}\int_{\mathcal{U}}\int_{\mathcal{Z}}\hspace{-.5em}\mathbb{1}_{B}(h_{k+1}^z)\mathcal{Q}^z_k(dz_{k+1}|z_k,u_k)\label{eq_proof_occupation_measure_pushforward_eq1}\\&\qquad\pi_k^{z}(du_k|h_k^z)p_k^z(dh_k^z) \nonumber \\
            &=\!\int_{\mathcal{H}_k}\int_{\mathcal{U}}\int_{\mathcal{X}}\hspace{-.5em}\mathbb{1}_{B}(\psi(h_{k+1}))\mathcal{Q}(dx_{k+1}|x_k,u_k)\label{eq_proof_occupation_measure_pushforward_eq2}\\&\qquad\pi_k^{z}(du_k|\psi(h_k))p_k^z(\psi(dh_k)) \nonumber
            \\
            &=\!\int_{\mathcal{H}_k}\!\int_{\mathcal{U}}\!\int_{\mathcal{X}}\hspace{-.5em}\mathbb{1}_{\psi^{-1}(B)}(h_{k+1})\mathcal{Q}(dx_{k+1}|x_k,u_k)\label{eq_proof_occupation_measure_pushforward_eq3}\\&\qquad\pi_k(du_k|h_k)p_k(dh_k) \nonumber
            \\
            &=\!p_{k+1}(\psi^{-1}(B)) \label{eq_proof_occupation_measure_pushforward_eq4},
        \end{align}
    \end{subequations}
    where the first two equalities follow from the definition of the occupation measure, the third equality follows from the change of variables $h_k^z=\psi(h_k)$ and uses the fact that $\mathcal{Q}^z_k(B\times\mathcal{C}|z_k,u_k) = \mathcal{Q}(B|x_k,u_k)$ for all $B\in\mathcal{B}(\mathcal{X})$ and for all $z_k$ and $x_k$ associated to $h_k^z=\psi(h_k)$, the fourth equality follows from the induction hypothesis and the definition of $\pi$, the last equality again follows by definition. 

    Starting the induction from $k=0$, where $p_0^z((x_0,0)) = p_0^z(\psi(x_0)) = p_0(x_0)=1$ for the initial state $x_0$ and zero otherwise, we conclude that $p_k^z = \psi_\# p_k$ for all $k\in[N]$.
\end{proof}
In other words, the probability of visiting a set of states and actions is equivalent in both MDPs $\mathcal{M}$ and $\mathcal{M}_z$ under the respective policies $\pi$ and $\pi^z$. Their occupation measure only differs in the additional cost state, which has following implications.
\begin{lem}
    Let $(a,\pi)$ be feasible for \eqref{eq_main_problem_formulation_using_expectation}, then the adaptation $\pi^z$ is feasible for \eqref{eq_constrained_mdp_formulation}, and vice versa.
    \label{lem_adaptations_feasible}
\end{lem}
\begin{proof}
    This follows immediately by Lemma \ref{lem_occupation_measure_pushforward}, which allows a change of variables in $\int_{\mathcal{H}_N}\mathbb{1}_{\mathcal{H}^s_N}(h)p_N(dh)= \int_{\mathcal{H}^z_N}\mathbb{1}_{\mathcal{H}^s_N}(\psi^{-1}(h^z))p_N^z(dh^z)$. If one term is greater than or equal to $\alpha$, so is the other.
\end{proof}

We are now ready for prove the existence of a solution for \eqref{eq_constrained_mdp_formulation}, and verify the solutions' optimality in \eqref{eq_main_problem_formulation_using_expectation}. For the former, we rely on the following.
\begin{lem}[{\cite[Theorem 5.1, Lemma 4.1]{feinberg2002nonatomic}}]\label{lem_optimality_of_constrained_MDPs}
    Assume that $\mathcal{Q}$ is a weakly continuous transition kernel, the input space $\mathcal{U}$ is compact, the reward functions $r_k^i$ nonpositive and upper semi-continuous and $\alpha_i\in\mathbb{R}$ for all $k\in\mathbb{N}$, $i\in[M]$, $N,M\in\mathbb{N}$. The constrained MDP
    \begin{subequations}
        \begin{align}
            &\sup_{\pi\in\Pi} &&\mathbb{E}_{x_0}^{\pi}\left[\sum_{k=0}^{\infty} r_k^0(x_k,u_k)\right] \\
            &\text{s.t.} && \mathbb{E}_{x_0}^{\pi}\left[\sum_{k=0}^{\infty} r_k^i(x_k,u_k)\right]\geq \alpha_i
        \end{align}%
        \label{eq_general_constrained_mdp}%
    \end{subequations}%
    admits a feasible policy if and only if it admits an optimal stochastic Markov policy.
\end{lem}

\begin{lem}
    Under Assumption \ref{ass_assumptions}, there exists a stochastic Markov policy $\pi\in\Pi$ that optimally solves Problem \eqref{eq_constrained_mdp_formulation}.    \label{lem_existence_and_markov_property_of_optimal_solution}
\end{lem}

\begin{proof}
    A feasible solution to Problem \eqref{eq_main_problem_formulation_using_expectation} exists by Assumption \ref{ass_assumptions}. Then, a feasible solution exists for Problem \eqref{eq_constrained_mdp_formulation} by Lemma \ref{lem_adaptations_feasible}. By Assumption \ref{ass_assumptions}, for all $k\in[N-1]$ the set $\mathcal{U}_k$ is compact and the kernel $\mathcal{Q}_k$ is weakly continuous since, for all $\gamma\in C_b(\mathcal{Z})$,
    \begin{subequations}
        \begin{align}
            &\int_{\mathcal{Z}}\gamma(z_{k+1})\mathcal{Q}^z_k(dz_{k+1}|z_k, u_k)
            \\ & \quad = \int_{\mathcal{X}}\int_{\mathcal{C}}\gamma((x_{k+1},c_{k+1}))
            \\ & \qquad\delta_{c_k+\ell_k(x_k,u_k)}(dc_{k+1})
             \mathcal{Q}(dx_{k+1}|x_k, u_k) \nonumber
            \\ & \quad = \int_{\mathcal{X}}\gamma((x_{k+1},c_k\!+\!\ell_k(x_k,u_k))\mathcal{Q}(dx_{k+1}|x_k, u_k) 
        \end{align}
    \end{subequations}
    is continuous and bounded since $\ell_k$ is continuous, $\gamma$ is continuous and bounded, and $\mathcal{Q}$ weakly continuous. The set $\mathcal{U}_N=\{0,1\}$ is compact and the kernel $\mathcal{Q}_N$ is weakly continuous by definition. Choosing $r_N^0=-f_N$, $r_N^1 = g_N-1$ and $r_k^0=r_k^1=0$ for $k\in\mathbb{N}\setminus\{N\}$ renders $r_k^0$ and $r_k^1$ nonpositive and upper-semicontinuous for all $k\in\mathbb{N}$ since $\mathcal{Z}\setminus x^{\dagger}=\mathcal{X}^s\times\mathcal{C}$ is closed. Choosing $\alpha_1 = \alpha - 1$ recovers the constraint of \eqref{eq_constrained_mdp_formulation_constraint} and renders Problems \eqref{eq_constrained_mdp_formulation} and \eqref{eq_general_constrained_mdp} equivalent. By Lemma \ref{lem_optimality_of_constrained_MDPs}, Problem \eqref{eq_constrained_mdp_formulation} admits an optimal stochastic Markov policy. 
\end{proof}

\begin{lem}
    Problem \eqref{eq_main_problem_formulation_using_expectation} and \eqref{eq_constrained_mdp_formulation} are equivalent in the following sense.
    \begin{enumerate}
        \item If $\pi^z$ is optimal for \eqref{eq_constrained_mdp_formulation}, then its adaptation $(\pi,a)$ to \eqref{eq_main_problem_formulation_using_expectation} optimally solves \eqref{eq_main_problem_formulation_using_expectation}. 
        
        \item If $(\pi,a)$ is optimal for \eqref{eq_main_problem_formulation_using_expectation}, then its adaptation $\pi^z$ to \eqref{eq_constrained_mdp_formulation} optimally solves \eqref{eq_constrained_mdp_formulation}.
    \end{enumerate}
    \label{lem_adaptations_optimal}
\end{lem} 
\begin{proof}
    Let $p^z_{0:N+1}$ be the occupation measure on $\mathcal{H}_{0:N+1}^z$ induced by $\pi^z$ on $\mathcal{M}_z$, and $p_{0:N}$ the occupation measure induced on $\mathcal{H}_{0:N}$ by $\pi$ on $\mathcal{M}$. For notational simplicity, we write $L(h_N)=\ell_N(x_N)+\sum_{k=0}^{N-1}\ell_k(x_k,u_k)$. As a consequence of Lemma \ref{lem_occupation_measure_pushforward}, by a change of variables $h_N^z=\psi(h_N)$, 
    \begin{subequations}
        \begin{align}
            &\int_{\mathcal{H}_{N+1}^z}(c_N+\ell_N(x_N))u_N p_{N+1}^z(dh^z_{N+1}) \label{eq_proof_of_implication_problems_align_eq21}
            \\
            &=\int_{\mathcal{H}_N^z}L(\psi^{-1}(h_N^z)) \pi_N^z(1|h^z_N)p_{N}^z(dh^z_{N}) \label{eq_proof_of_implication_problems_align_eq22}
            \\
            &=\int_{\mathcal{H}_N}L(h_N)\pi_N^z(1|\psi(h_N))p_N(dh_N). \label{eq_proof_of_implication_problems_align_eq23}
        \end{align}
        \label{eq_proof_of_implication_problems_align_eq14_all}
    \end{subequations}
    Further, by Lemma \ref{lem_adaptations_feasible}, the adaptations in 1) and 2) are feasible for \eqref{eq_main_problem_formulation_using_expectation} and \eqref{eq_constrained_mdp_formulation}, respectively. 
    
    Assume, for the sake of contradiction, that $(a,\pi)$ is suboptimal in \eqref{eq_main_problem_formulation_using_expectation}. Then, there exists $(a^{\star},\pi^{\star})$ feasible for \eqref{eq_main_problem_formulation_using_expectation} on $\mathcal{M}$ which achieves a lower objective value in \eqref{eq_main_problem_formulation_using_expectation_objective}. Let $\pi^{z,\star}$ be the adaptation of $(a^{\star},\pi^{\star})$ to \eqref{eq_constrained_mdp_formulation}. Denote the occupation measure induced by $\pi^{\star}$ on $\mathcal{H}_{0:N}$ by $p_{0:N}^{\star}$ and the occupation measure induced by $\pi^{z,\star}$ on $\mathcal{H}^z_{0:N+1}$ by $p_{0:N+1}^{z,\star}$. Then, 
    \begin{subequations}
        \begin{align}
            &\int_{\mathcal{H}_N}L(h_N)\pi_N^z(1|\psi(h_N))p_N(dh_N) \label{eq_proof_of_implication_problems_align_eq11}
            \\
            &=\int_{\mathcal{H}_N}L(h_N)a(h_N)p_N(dh_N) \label{eq_proof_of_implication_problems_align_eq12}
            \\
            &>\int_{\mathcal{H}_N}L(h_N)a^{\star}(h_N)p_N^{\star}(dh_N) \label{eq_proof_of_implication_problems_align_eq13}
            \\
            &=\int_{\mathcal{H}_N^z}L(\psi^{-1}(h_N^z))\pi_N^{z,\star}(1|h_N^z)p_N^{z,\star}(dh_N^z) \label{eq_proof_of_implication_problems_align_eq14}
            \\
            & \geq \int_{\mathcal{H}_N^z}L(\psi^{-1}(h_N^z))\pi_N^z(1|h_N^z)p_N^z(dh_N^z), \label{eq_proof_of_implication_problems_align_eq15}
        \end{align}
    \end{subequations}
    where \eqref{eq_proof_of_implication_problems_align_eq12} holds by a definition of the adaptation, \eqref{eq_proof_of_implication_problems_align_eq13} holds by the suboptimality assumption of $(a,\pi)$, \eqref{eq_proof_of_implication_problems_align_eq14} holds by a change of variables, and \eqref{eq_proof_of_implication_problems_align_eq15} by optimality of $\pi^z$ in \eqref{eq_constrained_mdp_formulation}. This contradicts, hence $(\pi,a)$ must be optimal for \eqref{eq_main_problem_formulation_using_expectation}.

    The proof of 2) follows analogously and is omitted in the interest of space.
\end{proof}
Theorem \ref{thm_strong_duality_main} now follows by combining Lemma \ref{lem_existence_and_markov_property_of_optimal_solution} and Lemma \ref{lem_adaptations_optimal}. 

\subsection{Comment: Theorem \ref{thm_strong_duality_main} with Compact $\mathcal{C}$.}
\label{sec_appendix_comment}
The proof of Theorem \ref{thm_strong_duality_main} is fully analogous for a bounded cost space $\mathcal{C}$. Any cost-state appended to a history $h_k\in\mathcal{H}_k$ by $\psi(h_k)$ is within $\mathcal{C}$ by definition. It only remains to show that the kernels $Q^z_{0:N}$ are weakly continuous to ensure validity of Lemma \ref{lem_optimality_of_constrained_MDPs}.

To simplify notation, we define the function $m(c_k,x_k,u_k)=\min(C_{\text{max}},c_k\!+\!\ell_k(x_k,u_k))$, which is continuous in its arguments since $\ell_k$ is continuous for all $k\in[N]$ by Assumption \ref{ass_assumptions}. Let $\gamma\in C_b(\mathcal{Z})$. Then, the composition $\gamma((x_{k+1}, m(c_k,x_k,u_k)))$ is continuous in $(x_{k+1},x_k,c_k,u_k)$ and bounded. Further,  
\begin{subequations}
    \begin{align}
        &\int_{\mathcal{Z}}\gamma(z_{k+1})\mathcal{Q}^z_k(dz_{k+1}|z_k, u_k)
        \\ & \quad = \int_{\mathcal{X}}\int_{\mathcal{C}}\gamma((x_{k+1},c_{k+1}))
        \\ & \qquad\delta_{m(c_k,x_k,u_k)}(dc_{k+1})
         \mathcal{Q}(dx_{k+1}|x_k, u_k) \nonumber
        \\ & \quad = \int_{\mathcal{X}}\gamma((x_{k+1},m(c_k,x_k,u_k))\mathcal{Q}(dx_{k+1}|x_k, u_k) 
    \end{align}
\end{subequations}
is continuous in $(z_k,u_k)$ with $z_k=(x_k,c_k)$ and bounded since $\mathcal{Q}$ is weakly continuous in $(x_k,u_k)$ by Assumption \ref{ass_assumptions}. 
Hence, $\mathcal{Q}^z_k$ is weakly continuous for all $k\in[N-1]$; $\mathcal{Q}^z_N$ is continuous by definition. 

\subsection{Proof of Theorem \ref{prop_error_bounds}} \label{sec_appendix_proof_of_error_bounds}
Similarly to $\mathcal{M}^z$, we introduce a gridding abstraction for $\mathcal{M}$, which will act as a bridge between $\mathcal{M}$ and $\mathcal{M}^{rz}$ in our derivation. We denote this abstraction by $\mathcal{M}^r=\{\mathcal{X}^r, {\mathcal{U}}^r,\mathcal{Q}^r,{\ell}_{0:N},g_{0:N}, \alpha\}$, where $\mathcal{Q}^r({x}^r_j|{x},{u})=\mathcal{Q}(\mathcal{X}_j|x,u)$ for all $j=1,\mydots,M_x$. The different MDPs are summarized in Table \ref{tab_mdp_relations}.

\begin{table}[!tb]
    \centering
    \begin{tabular}{c|cccc}
         MDP & State space & Action space & Kernel & $k$-step-histories \\ \hline
         $\mathcal{M}$ & $\mathcal{X}$ & $\mathcal{U}$ & $\mathcal{Q}$ & $\mathcal{H}_k$ \\
         $\mathcal{M}^z$ & $\mathcal{Z}\!=\!\mathcal{X}\!\times\!\mathcal{C}$ &  $\mathcal{U}_k$ & $\mathcal{Q}^z_k$ & $\mathcal{H}^z_k$ \\
         $\mathcal{M}^r$ & $\mathcal{X}^r$ & $\mathcal{U}^r$ & $\mathcal{Q}^r$ & $\mathcal{H}^r_k$ \\
         $\mathcal{M}^{rz}$ & $\mathcal{Z}^r\!=\!\mathcal{X}^r\!\times\!\mathcal{C}^r$ & $\mathcal{U}^r_k$ & $\mathcal{Q}^{rz}_k$ & $\mathcal{H}^{rz}_k$
    \end{tabular}
    \caption{Overview of the MDPs that are used in the proof of Theorem \ref{prop_error_bounds}. $\mathcal{M}^z$ differs from $\mathcal{M}$ by the accumulation of costs in an additional cost state and an additional time-step $N+1$. The MDPs $\mathcal{M}^r$ and $\mathcal{M}^{rz}$ are finite gridding abstractions of $\mathcal{M}$ and $\mathcal{M}^z$, respectively.}
    \label{tab_mdp_relations}
\end{table}

The proof is structured as follows: 1. We define how a policy $\pi^{rz}$ associated to $\mathcal{M}^{rz}$ is executed on $\mathcal{M}$ and $\mathcal{M}^r$. We show that $\pi^{rz}$ achieves a safety of $\alpha$ on $\mathcal{M}^r$ if it is feasible for \eqref{eq_constrained_mdp_formulation} on $\mathcal{M}^{rz}$. 2. We bound the difference of the safety achieved by $\pi^{rz}$ on $\mathcal{M}$ and $\mathcal{M}^r$. 3. We combine these results to prove Theorem \ref{prop_error_bounds}.

\textbf{Step 1: Safety on $\bm{\mathcal{M}^r}$.} We define how a policy $\pi^{rz}$ associated to $\mathcal{M}^{rz}$ is executed on the MDPs $\mathcal{M}$ and $\mathcal{M}^r$. For this, we require some auxiliary notation. By abuse of notation, we denote by $\xi$ the following maps: If $x\in{\mathcal{X}}_j$, then $\xi({x})={x}_j^r$, while $\xi(x=x^{\dagger})=x^{\dagger}$; if $z\in\mathcal{Z}_j$, then $\xi(z)=z_j^r$; if $u\in\mathcal{U}_j$, then $\xi(u)=u_j^r$; if $c\in C_j$, then $\xi(c)=c_j^r$; and finally for any $h_k\in\mathcal{H}_k$ with $k\in[N]$, $\xi(h_k)=(\xi(x_0),\xi(u_{0}),\mydots,\xi(x_k))$. We denote by $\zeta:\mathcal{H}_k\rightarrow\mathcal{C}^r$ the cost of a history $h_k\in\mathcal{H}_k$ propagated through the gridded dynamics of $\mathcal{M}^{r}$, which is recursively defined by 
\begin{align*}
    \zeta(h_k) = \xi\Big(\zeta(h_{k-1}) + \ell_{k-1}\big(\xi(x_{k-1}),\xi(u_{k-1})\big)\Big),
\end{align*}
with $\zeta(h_0)=0$. The mapping $\psi^r:\mathcal{H}_k\rightarrow\mathcal{H}_k^{rz}$ discretizes the history and augments the discretized cost to every state, 
\begin{align*}
    \psi^r(h_k)=\Big(&(\xi(x_k),\zeta(h_k)),\xi(u_k),\mydots,(\xi(x_0),\zeta(h_0))\Big).
\end{align*}
The functions $\zeta$ and $\psi^r$ are extended to all $k\in[N]$.

\begin{defn}\label{def_policy_adaptation_r}
    Let $\pi^{rz}=\{\pi_0^{rz},\dots,\pi_N^{rz}\}$ with $\pi_k^{rz}:\mathcal{B}(\mathcal{U})\times\mathcal{H}^{rz}_k\rightarrow[0,1]$ for $k\in[N]$ be a policy associated to $\mathcal{M}^{rz}$. We call $\pi=\{\pi_0,\mydots,\pi_{N-1}\}$ with $\pi_k(\cdot|h_k)=\pi_k^{rz}(\cdot|\psi^r(h_k))$ for all $h_k\in\mathcal{H}_k$, $k\in[N-1]$, the adaptation of $\pi^{rz}$ to $\mathcal{M}$ and $\mathcal{M}^r$.
\end{defn}
The adaptation in Definition \ref{def_policy_adaptation_r} allows to execute any policy $\pi^{rz}$ associated to $\mathcal{M}^{rz}$ on both $\mathcal{M}$ and $\mathcal{M}^r$. Further, it leads to a counterpart of Lemma \ref{lem_occupation_measure_pushforward} for the MDPs $\mathcal{M}^r$ and $\mathcal{M}^{rz}$.
\begin{lem}
    Let $\pi^{rz}$ be associated to $\mathcal{M}^{rz}$ and $\pi$ be the respective adaptation to $\mathcal{M}^r$. Let $p_{0:N+1}^{rz}$ be the occupation measures on $\mathcal{H}^{rz}_{0:N+1}$ induced by executing $\pi^{rz}$ on $\mathcal{M}^{rz}$ and $p_{0:N}^r$ be the occupation measures on $\mathcal{H}^r_{0:N}$ induced by executing $\pi^{r}$ on $\mathcal{M}^r$. Then, $p_k^{rz} = \psi^r_\# p_k^r$ for all $k\in[N]$, where $\psi^r_\#$ denotes the push-forward through $\psi^r$.
    \label{lem_occupation_measure_pushforward_r}
\end{lem}
Lemma \ref{lem_occupation_measure_pushforward_r} is equivalent to Lemma \ref{lem_occupation_measure_pushforward} applied to $\mathcal{M}^{rz}$ and $\mathcal{M}^r$; the proof is analogous.
\begin{cor}
    The policy $\pi^{rz}$ has the same safety on $\mathcal{M}^{rz}$ as its adaptation to $\mathcal{M}^r$ on $\mathcal{M}^r$; if $\pi^{rz}$ is feasible for \eqref{eq_constrained_mdp_formulation}, then applying its adaptation to $\mathcal{M}^r$ on $\mathcal{M}^r$ yields a safety of at least $\alpha$.
    \label{cor_same_safety_Mrz_Mr}
\end{cor}

\textbf{Step 2: Safety Difference between $\bm{\mathcal{M}^r}$ and $\bm{\mathcal{M}}$.}
We will now prove that the difference of safety in executing any adaptation of $\pi^{rz}$ on $\mathcal{M}^r$ and on $\mathcal{M}$ is bounded by the expression in Proposition \ref{prop_error_bounds}. Then, based on Lemma \ref{cor_same_safety_Mrz_Mr}, also the difference of safety between $\mathcal{M}^{rz}$ and $\mathcal{M}$ is bounded.

We first show that the adapted policy only depends on the sequence of visited state and action partitions.
\begin{lem}
    Let $\pi^{rz}=\{\pi_0^{rz},\dots,\pi_N^{rz}\}$ be a policy associated to $\mathcal{M}^{rz}$, which is adapted to $\mathcal{M}$ and $\mathcal{M}^r$ through $\pi=\{\pi_0,\dots,\pi_{N-1}\}$. For any $k\in[N-1]$ and histories $h_k,h_k'\in\mathcal{H}_k$ with $\xi(h_k)=\xi(h_k')$ it holds that $\pi_k(\cdot|h_k)=\pi_k(\cdot|h_k')$.
    \label{lem_same_history_same_policy_r}
\end{lem}
\begin{proof}
    This follows immediately by definition of $\zeta$ and $\psi^r$ as well as Definition \ref{def_policy_adaptation_r}: Given $\xi(h_k)=\xi(h_k')$, both histories are associated to the same state and action partitions. Therefore, $\ell(\xi(x_j),\xi(u_j))=\ell(\xi(x_j'),\xi(u_j'))$ for all $j<k$, and the propagated cost through the discretized dynamics is the same, $\zeta(h_k)=\zeta(h_k')$. Further, $\psi^r(h_k)=\psi^r(h_k')$. Then, $\pi_k(\cdot|h_k)=\pi_k^{rz}(\cdot|\psi^r(h_k))=\pi_k^{rz}(\cdot|\psi^r(h_k'))=\pi_k(\cdot|h_k')$.
\end{proof}
\begin{lem}[Theorem 1 in \cite{abate2008probabilistic}]
    Given the MDP $\mathcal{M}$ and a policy $\pi\in\Pi$, let $V_k:\mathcal{H}_k\rightarrow[0,1]$, $k\in[N]$, be recursively defined by
    \begin{subequations}
        \begin{align}
            V_N(h_N) &= 1 \\
            V_k(h_k) &= \int_{\mathcal{U}}\int_{\mathcal{X}^s}V_{k+1}(h_{k+1})\mathcal{Q}(dx_{k+1}|x_k,u_k)\pi_k(du_k|h_k)
        \end{align}%
        \label{eq_dp_recursion_safety_eval}%
    \end{subequations}%
    for all ${h}_k\in{\mathcal{H}}_k^s$, $k\in[N-1]$, ${h}_N\in{\mathcal{H}}_N^s$; for all $k\in[N]$ let $V_k(h_k)=0$ if $h_k\notin\mathcal{H}_k^s$. Then, $V_0(h_0)=\mathbb{P}_{x_0}^{\pi}(x_{0:N}\in \mathcal{X}^{s})$. 
    \label{lem_invariance_probability_computation}
\end{lem}
The safety of a policy on $\mathcal{M}^r$ can be computed equivalently. We denote the corresponding value function by $V_k^r$.
\begin{lem}
    \label{prop_same_history_same_safety}
    Let $h_k,h_k'\in\mathcal{H}_k$ be two histories with $\xi(h_k)=\xi(h_k')$ and $x_k=x_k'$ for some $k\in[N]$. Let a policy $\pi^{rz}$ associated to $\mathcal{M}^{rz}$ be adapted to $\mathcal{M}$ through $\pi=\{\pi_0,\dots,\pi_{N-1}\}$. Let $V_k$ be the value functions in recursion \eqref{eq_dp_recursion_safety_eval} associated to $\mathcal{M}$ under policy $\pi$. Then, $V_k(h_k)=V_k(h_k')$ for all $k\in[N]$.
\end{lem}
\begin{proof}
    The proof uses the induction hypothesis that $V_{k+1}(h_{k+1})=V_{k+1}(h_{k+1}')$ for all $h_{k+1},h_{k+1}'\in\mathcal{H}_{k+1}$ with $\xi(h_{k+1})=\xi(h_{k+1}')$ and $x_{k+1}=x_{k+1}'$ for some $k\in[N-1]$. Then, for any $h_k,h_k'\in\mathcal{H}_{k}$ with $\xi(h_k)=\xi(h_k')$ and $x_k=x_k'$, recursion \eqref{eq_dp_recursion_safety_eval} yields
    \begin{align*}
        V_k(h_k) &= \int_{\mathcal{U}}\int_{\mathcal{X}^s}V_{k+1}(h_{k+1})\mathcal{Q}(dx_{k+1}|x_k,u_k){\pi}_k(du_k|h_k)
        \\
        &= \int_{\mathcal{U}}\int_{\mathcal{X}^s}V_{k+1}(h_{k+1}')\mathcal{Q}(dx_{k+1}|x_k',u_k){\pi}_k(du_k|h_k')
        \\
        &= V_k(h_k'),
    \end{align*}
    since $x_k=x_k'$, $\pi_k(\cdot|h_k)=\pi_k(\cdot|h_k')$ by Lemma \ref{lem_same_history_same_policy_r}, and $V_{k+1}(h_{k+1})=V_{k+1}(h_{k+1}')$ when $x_{k+1}=x'_{k+1}$ by the induction hypothesis. Note that the partitions $\mathcal{X}_i$, $i=1,\dots,M_x$ are subsets of $\mathcal{X}^s$. Therefore, if $\xi(h_k)=\xi(h_k')$ and $h_k\in\mathcal{H}^s_k$, then also $h'_k\in\mathcal{H}^s_k$. Starting the induction at $k=N$, either $V_N(h_N)=V_N(h_N')=1$ or $V_N(h_N)=V_N(h_N')=0$ depending on whether $h_N\in\mathcal{H}^s_N$ or not, which satisfies the induction hypothesis.
\end{proof}
Under Assumption \ref{ass_lipschitz_cont}, the value function $V_k$ on $\mathcal{M}$ satisfies a Lipschitz condition within every partition.
\begin{lem}
    Let a policy $\pi^{rz}$ associated to $\mathcal{M}^{rz}$ be adapted to $\mathcal{M}$ through $\pi=\{\pi_0,\dots,\pi_{N-1}\}$. Let $V_k$ be the value functions in \eqref{eq_dp_recursion_safety_eval} associated to $\mathcal{M}$ under policy $\pi$. Under Assumption \ref{ass_lipschitz_cont}, for all $k\in[N]$ and $h_k,h_k'\in\mathcal{H}_k$ with $\xi(h_k)=\xi(h_k')$, it holds that
    \begin{align*}
        |V_k(h_k) - V_k(h_k')| &\leq M_x\overline{\mu}h_x\Delta_x
    \end{align*}
    \label{lem_value_function_lipschitz}
\end{lem}\allowdisplaybreaks
\begin{proof}
    Let $\mu$ be the Lebesgue measure and $\mu_j = \mu(\mathcal{X}_j)$ for $j=1,\mydots,M_x$. Note that, by construction, $V_k(\cdot)\in[0,1]$. Then,
    \begin{subequations}
        \begin{align}
            &|V_k(h_k)-V_k(h_k')| 
            \\ & = \bigg|\int_{\mathcal{U}}\!\sum_{j=1,\dots,M_x}\!\int_{\mathcal{X}_j}\!V_{k+1}(h_{k+1})\mathcal{Q}(dx_{k+1}|x_k,u_k){\pi}_k(du_k|h_k) \nonumber 
            \\& \quad -\!\int_{\mathcal{U}}\!\sum_{j=1,\dots,M_x}\!\int_{\mathcal{X}_j}\hspace{-.4em}\!V_{k+1}(h'_{k+1})\mathcal{Q}(dx_{k+1}|x_k',u_k){\pi}_k(du_k|h_k')\bigg| 
            \label{eq_proof_error_bounds_V_lipschitz1}\\
            & \leq \int_{\mathcal{U}}\sum_{j=1,\dots,M_x}\int_{\mathcal{X}_j}|V_{k+1}(h_{k+1})||q(x_{k+1}|x_k,u_k) 
            \label{eq_proof_error_bounds_V_lipschitz2}
            \\ & \quad -q(x_{k+1}|x_k',u_k)|dx_{k+1}|{\pi}_k(du_k|h_k)| \nonumber \\
            & \leq \int_{\mathcal{U}}\sum_{j=1,\dots,M_x}\int_{\mathcal{X}_j} |q(x_{k+1}|x_k,u_k)
            \label{eq_proof_error_bounds_V_lipschitz3}
            \\ & \quad -q(x_{k+1}|x_k',u_k)|dx_{k+1}|{\pi}_k(du_k|h_k)| \nonumber \\
            &\leq \int_{\mathcal{U}}\sum_{j=1,\dots,M_x}\mu_j h_x||x_k-x_k'|||{\pi}_k(du_k|h_k)|
            \label{eq_proof_error_bounds_V_lipschitz4}\\
            &\leq M_x\overline{\mu} h_x||x_k-x_k'||
            \label{eq_proof_error_bounds_V_lipschitz5}\\
            &\leq M_x\overline{\mu}h_x\Delta_x,\label{eq_proof_error_bounds_V_lipschitz6}
        \end{align}
    \end{subequations}
    where \eqref{eq_proof_error_bounds_V_lipschitz1} holds by definition of $V_k$, \eqref{eq_proof_error_bounds_V_lipschitz2} by the triangle inequality, Lemma \ref{lem_same_history_same_policy_r} and $V_{k+1}(h_{k+1})=V_{k+1}(h_{k+1}')$ when $x_{k+1}=x'_{k+1}$ by Lemma \ref{prop_same_history_same_safety}, \eqref{eq_proof_error_bounds_V_lipschitz3} by $|V_{k+1}(\cdot)|\leq 1$, \eqref{eq_proof_error_bounds_V_lipschitz4} by Assumption \ref{ass_lipschitz_cont}, \eqref{eq_proof_error_bounds_V_lipschitz5} by definition of $\overline{\mu}$ and the fact that $\pi_k(\cdot|h_k)$ is a probability measure, and \eqref{eq_proof_error_bounds_V_lipschitz6} by definition of the maximum diameter $\Delta_x$ of the state partitions and $\xi(h_k)=\xi(h_k')$.
\end{proof}
The fact that $V_k$ is Lipschitz also provides a bound on the difference to $V^r_k$ for every $k\in[N]$.
\begin{lem}
    \label{lem_boundedness_of_error_to_Mr}
    Let a policy $\pi^{rz}$ associated to $\mathcal{M}^{rz}$ be adapted to $\mathcal{M}$ and $\mathcal{M}^r$ through $\pi=\{\pi_0,\dots,\pi_{N-1}\}$. Let $V_k$ and $V_k^r$ be the value functions in \eqref{eq_dp_recursion_safety_eval} associated to $\mathcal{M}$ and $\mathcal{M}^r$ under policy $\pi$, respectively. Under Assumption \ref{ass_lipschitz_cont}, for all $k\in[N]$ and all $h_k\in\mathcal{H}_k^s$, it holds that 
    \begin{align*}
        |V_k(h_k)-V^r_k(\xi(h_k))|\leq (N-k)h_xM_x\overline{\mu}\Delta_x.
    \end{align*}
\end{lem}
\begin{proof}    
    Let $\hat{V}^r_k:\mathcal{H}_k\rightarrow[0,1]$ denote the piece-wise constant extension of the function $V_k^{r}$ to the continuous space $\mathcal{H}_k$,  $\hat{V}^r_k(h_k)=V_k^{r}(\xi(h_k))$, for all $k\in[N]$. In the same vein as in \cite{abate2010approximate}, we show by induction that the difference between the value functions $V_k$ and $\hat{V}_k^{r}$ is bounded for all $k\in[N]$. For this purpose, we require following preliminary result: For any $h_k\in\mathcal{H}_k$, $k\in[N-1]$, $j=1,\dots,M_x$, $i=1,\dots,M_u$,
    \begin{subequations}%
        \label{eq_proof_error_bounds_prelim_for_step_2}%
        \begin{align}
            &\hat{V}_{k+1}^r({h}_{k+1})\mathcal{Q}^r({x}_j^r|x_k,u^r_i){\pi}_k(u^r_i|h_{k}) 
            \\ 
            &= \int_{\mathcal{U}_i}\int_{\mathcal{X}_j}\hat{V}^r_{k+1}({h}_{k+1})
            \mathcal{Q}(dx_{k+1}|x_k,u_k){\pi}_k(du_k|h_{k}),
        \end{align}
    \end{subequations}
    since, for any $h_k=(x_k,u_{k-1},\dots,x_0)$, $\hat{V}^r_{k+1}({h}_{k+1})$ is constant across all $h_{k+1}=(x_{k+1},u_k,\dots,x_0)$ with $x_{k+1}\in\mathcal{X}_j$ and $u_k\in\mathcal{U}_i$ by definition, $\mathcal{Q}^r({x}_j^r|x_k,{u}_k)=\mathcal{Q}(\mathcal{X}_j|x_k,{u}_k)$ by definition and $\pi_k(u_i^r|h_{k})=\pi_k(\mathcal{U}_i|h_k)$ is concentrated at the representative inputs by construction of the adaptation.

    Adding a zero and applying the triangle inequality, we obtain for all $h_k\in\mathcal{H}_k$,
    \begin{align*}
        &|V_k(h_k)-\hat{V}^r_k(h_k)|=|V_k(h_k)-\hat{V}^{r}_k(\xi(h_k))|\\
        &\leq |V_k(h_k) - V_k(\xi(h_k))| + |V_k(\xi(h_k))-{V}^{r}_k(\xi(h_k))|. 
    \end{align*}
    Let $\mathcal{K}=M_xh_x\overline{\mu}\Delta_x$. Based on Lemma \ref{lem_value_function_lipschitz}, we bound the first term by $|V_k(h_k)-V_k(\xi(h_k))|\leq\mathcal{K}$. Next, denoting $h_k=(x_k,u_{k-1},\mydots, x_0)\in\mathcal{H}_k$, ${h}_{k+1}=({x}_{k+1}, {u}_k, \xi(x_k),\xi(u_k),\mydots, \xi(x_0))\in\mathcal{H}_{k+1}$, $k\in[N-1]$, we obtain \allowdisplaybreaks
    \begin{align*}
        &|V_k(\xi(h_k))-{V}^{r}_k(\xi(h_k))| \\ \quad &= \bigg|\int_{\mathcal{U}}\int_{\mathcal{X}}V_{k+1}(h_{k+1})
         \mathcal{Q}(dx_{k+1}|\xi(x_k),u_k){\pi}_k(du_k|\xi(h_{k})) 
        \\&\quad - \sum_{u_k\in\mathcal{U}^r}\sum_{x_{k+1}\in\mathcal{X}^r}{V}^r_{k+1}({h}_{k+1})
        \\ &\quad \mathcal{Q}^r({x}_{k+1}|\xi(x_k),{u}_k){\pi}_k({u}_k|\xi(h_{k}))\bigg| 
        \\        
        &= \bigg|\int_{\mathcal{U}}\int_{\mathcal{X}}V_{k+1}(h_{k+1})
         \mathcal{Q}(dx_{k+1}|\xi(x_k),u_k){\pi}_k(du_k|\xi(h_{k}))
        \\& \quad-\!\int_{\mathcal{U}}\!\int_{\mathcal{X}}\!\hspace{-.4em}\hat{V}^r_{k+1}(h_{k+1})
        \mathcal{Q}(dx_{k+1}|\xi(x_k),\!u_k){\pi}_k(du_k|\xi(h_{k}))\bigg|
        \\
        &\leq \int_{\mathcal{U}}\int_{\mathcal{X}}|V_{k+1}(h_{k+1}) - \hat{V}^r_{k+1}(h_{k+1})|
        \\&\quad |\mathcal{Q}(dx_{k+1}|\xi(x_k),u_k)||{\pi}_k(du_k|\xi(h_{k}))|\\
        &\leq (N-k-1)\mathcal{K},
    \end{align*} 
\interdisplaylinepenalty=10000
    where the first equality follows by definition, the second equality holds by the preliminary result \eqref{eq_proof_error_bounds_prelim_for_step_2} and all partitions being disjoint but covering $\mathcal{X}$, the third relation by triangular inequality and the final relation by the induction hypothesis. Together,
    \begin{align*}
        |V_k(h_k)-\hat{V}^r_k(h_k)|\leq\mathcal{K}+ (N-k-1)\mathcal{K} = (N-k)\mathcal{K}, 
    \end{align*}
    which satisfies the induction hypothesis. 
    
    Note that, by construction of the state partition and definition of the value functions in \eqref{eq_dp_recursion_safety_eval}, $V_N(h_N)=\hat{V}_N^r(h_N)=1$ for all $h_N\in\mathcal{H}^{s}_N$ and zero otherwise. Hence, starting from time-step $k=N$, where $|V_N(h_N)-\hat{V}_N(h_N)|= 0$, completes the induction. 
\end{proof}

\textbf{Step 3: Conclusion.}
We are now in position to prove Theorem \ref{prop_error_bounds}: Let $\pi^{rz}$ be a policy associated to $\mathcal{M}^{rz}$ and adapted to $\mathcal{M}$ and $\mathcal{M}^r$ through $\pi$. By Lemma \ref{lem_occupation_measure_pushforward_r}, $\pi^{rz}$ achieves the same safety on $\mathcal{M}^{rz}$ as $\pi$ on $\mathcal{M}^r$. Since $\pi^{rz}$ is feasible for \eqref{eq_constrained_mdp_formulation} for the MDP $\mathcal{M}^{rz}$, this safety must be at least $\alpha$. By Lemma \ref{lem_boundedness_of_error_to_Mr}, the difference of safety achieved by $\pi$ on $\mathcal{M}$ and $\mathcal{M}^r$ is bounded by $|V_0(h_0)-V_0^r(\xi(h_0))|\leq N h_xM_x\overline{\mu}\Delta_x$. Therefore, the safety achieved by $\pi$ on $\mathcal{M}$ must be at least $\alpha - N h_xM_x\overline{\mu}\Delta_x$, which yields the claim.
\qedsymbol

\subsection{Proof of Proposition \ref{prop_constrained_mdp_formulation_refining}}
\label{sec_appendix_attainability_refined}
Rewrite \eqref{eq_constrained_mdp_formulation_refining} as
\begin{subequations}
    \begin{align}
        &\sup_{\pi\in\Pi} &&\mathbb{E}_{z_0}^{\pi}\left[g_{N+1}(z_{N+1})+\sum_{k=0}^{N}g_k(z_k,u_k)\right] \\
        &\text{s.t.} &&\mathbb{E}_{z_0}^{\pi} \left[f_{N+1}(z_{N+1})
        +\sum_{k=0}^{N}f_k(z_k,u_k)\right]\leq c^{\star}\label{eq_prop3proof_prefinaleq}\\
        & &&\mathbb{E}_{z_0}^{\pi} \left[f_{N+1}(z_{N+1})
        +\sum_{k=0}^{N}f_k(z_k,u_k)\right]\geq c^{\star}. 
    \end{align}%
    \label{eq_constrained_mdp_formulation_refining_inequalities}%
\end{subequations}%
Note that a feasible policy exists by definition, since $c^{\star}$ is the minimum of \eqref{eq_constrained_mdp_formulation}. Further, the functions $g_k$ are upper-semicontinuous and the functions $f_k$ continuous for all $k\in[N+1]$. Following the proof of Lemma \ref{lem_existence_and_markov_property_of_optimal_solution}, under Assumption \ref{ass_assumptions}, the transition kernel $\mathcal{Q}_k^z$ is weakly continuous, the input space $\mathcal{U}_k$ is compact, and \eqref{eq_constrained_mdp_formulation_refining_inequalities} can be brought into form \eqref{eq_general_constrained_mdp}, particularly by multiplying \eqref{eq_prop3proof_prefinaleq} with minus one. Then, by Lemma \ref{lem_optimality_of_constrained_MDPs} there exists a stochastic Markov policy optimally solving Problem \eqref{eq_constrained_mdp_formulation_refining}. \qedsymbol

\end{document}